\documentclass[11pt,a4paper]{amsart}
\usepackage{amsmath,amsfonts,amssymb,amscd,amsthm,tikz,comment,array}
\usepackage[shortlabels]{enumitem}
\usetikzlibrary{arrows, arrows.meta}

\usepackage[alphabetic]{amsrefs}

\usepackage{color, soul}
\definecolor{gr}{rgb}{0.7, 1, 0.7}
\definecolor{rr}{rgb}{1, 0.7, 0.7}

\usepackage{latexsym}
\usepackage{graphicx}

\theoremstyle{plain} %definition remark
\newtheorem{theorem}{Theorem}[section]

\newtheorem{lemma}[theorem]{Lemma}

\theoremstyle{definition} %definition remark
\newtheorem{definition}[theorem]{Definition}
\theoremstyle{remark} %definition remark
\newtheorem{remark}[theorem]{Remark}

\renewcommand{\mathfrak}{\mathbf}

\renewcommand{\Im}{\,\mathrm{Im}\,}
\renewcommand{\Re}{\,\mathrm{Re}\,}

\newcommand{\maxmin}{\,\mathrm{maxmin}\,}
\newcommand{\maxmax}{\,\mathrm{maxmax}\,}

\newcommand{\ignore}[1]{}

\newcommand{\C}{\mathbb{C}}

\newcommand{\N}{\mathbb{N}}
\newcommand{\R}{\mathbb{R}}

\newcommand{\diam}{\operatorname{diam}}

\renewcommand{\Im}{\operatorname{Im}}

\newcommand{\feig}{\mathrm{\,Feig}}
\newcommand{\hd}{\mathrm{dim_H}}
\newcommand{\hypd}{\mathrm{dim_{hyp}}}

\usepackage{amsmath}
\usepackage{mathtools}

\def\mid{\operatorname{mid}}
\def\rad{\operatorname{rad}}
\newcommand{\calP}{\mathcal{P}}

\newcommand{\ivC}{\textbf{\textsl{c}}}

\newcommand{\ivM}{\textbf{\textsl{m}}}

\newcommand{\ivMM}{\textbf{\textsl{M}}}

\newcommand{\ivW}{\textbf{\textsl{w}}}
\newcommand{\ivX}{\textbf{\textsl{x}}}

\newcommand{\ivZ}{\textbf{\textsl{z}}}

\newcommand{\lo}[1]{\underline{#1}}
\newcommand{\hi}[1]{\overline{#1}}

\title[]{Lower bounds on the Hausdorff dimension of some Julia sets}

\author{Artem Dudko}
\address{Institute of Mathematics of Polish Academy of Sciences, Warsaw, Poland}
\email{adudko@impan.pl}

\author{Igors Gorbovickis}
\address{Jacobs University, Bremen, Germany}
\email{i.gorbovickis@jacobs-university.de}

\author{Warwick Tucker}
\address{School of Mathematics, Monash University, Melbourne, Australia}
\email{Warwick.Tucker@monash.edu}

\subjclass[2010]{}
\date{\today}

\begin{document}
\begin{abstract}
We present an algorithm for a rigorous computation of lower bounds on the Hausdorff dimensions of Julia sets for a wide class of holomorphic maps. 
	We apply this algorithm to obtain lower bounds on the Hausdorff dimension of the Julia sets of some infinitely renormalizable real quadratic polynomials, including the Feigenbaum polynomial $p_\feig(z)=z^2+c_{\feig}$.
	In addition to that, we construct a piecewise constant function on $[-2,2]$ that provides rigorous lower bounds for the Hausdorff dimension of the Julia sets of all quadratic polynomials $p_c(z) = z^2+c$ with $c \in [-2,2]$.
	Finally, we verify the conjecture of Ludwik Jaksztas and Michel Zinsmeister that the Hausdorff dimension of the Julia set of a quadratic polynomial $p_c(z)=z^2+c$, is a $C^1$-smooth function of the real parameter $c$ on the interval $c\in(c_{\feig},-3/4)$. 
\end{abstract}
\maketitle

	\section{Introduction}
\label{sec:introduction}

The Julia set of a rational map on the Riemann sphere can be defined as the set of all its Lyapunov unstable points. Informally speaking, this is the set of points near which the system behaves chaotically, i.e. has unpredictable long term behavior.
A natural question is how big and complex this set can be.
One of the classical measures of geometric complexity of a Julia set (or any other fractal set) is its Hausdorff dimension.

The study of the Hausdorff dimension of Julia sets of rational maps in general, and of quadratic polynomials in particular, has a long history.
It was proven that for large classes of rational maps, including the hyperbolic and parabolic maps, the Julia set has Hausdorff dimension strictly less than two (e.g., see \cite{Urbanski_94, McMullen-HDII-00, Graczyk_Smirnov_98, Graczyk_Smirnov_09}). At the same time there exist rational maps whose Julia sets are proper subsets of the Riemann sphere and have Hausdorff dimension equal to~2. For instance, M.~Shishikura showed that such maps are generic in the family of all quadratic polynomials $p_c(z)=z^2+c$ with the parameter $c$ on the boundary of the Mandelbrot set~\cite{Shishikura-HDM-98}.

In what follows, for every $c\in\C$, we let $p_c\colon\C\to\C$ denote the quadratic polynomial
$$
p_c(z) = z^2+c.
$$
The Julia set of $p_c$ will be denoted by $J_c$, and $\hd(J_c)$ will stand for its Hausdorff dimension.

The first examples of quadratic polynomials $p_c$ with the Julia set of positive Lebesgue measure (and hence, also with $\hd(J_c)=2$) were constructed by X.~Buff and A.~Ch\'eritat~\cite{Buff_Cherita_pos_area}.
Other interesting examples of quadratic polynomials were constructed by A.~Avila and M.~Lyubich in the sequence of papers \cite{AvilaLyubich-SmallHD-06, Avila_Lyubich_08, AvilaLyubich-Area-22}. Namely, they showed that there exist Feigenbaum maps (i.e., infinitely renormalizable quadratic polynomials $p_c$ with bounded combinatorics and a priori bounds) of the following three types: (i) $\hd(J_c)<2$; (ii) $\hd(J_c)=2$ and the Lebesgue measure of $J_c$ is zero; (iii) the Lebesgue measure of $J_c$ is strictly positive.

We note that even though, according to M.~Shishikura~\cite{Shishikura-HDM-98}, there are plenty of complex parameters $c\in\C$, for which the Julia set $J_c$ has the highest possible Hausdorff dimension (i.e., $\hd(J_c)=2$), it is not known, whether any of these parameters can be real. In fact, not much is known about how large the dimension $\hd(J_c)$ can be when the parameter $c$ is restricted just to the real line: it was shown in~\cite{Levin_Zin_13} that there exist $c\in\R$, for which $\hd(J_c)>4/3$, and this seems to be the highest lower bound that has been rigorously confirmed. Our first result improves this lower bound.

We recall that for a rational map $f\colon\hat{\C}\to\hat{\C}$, the hyperbolic dimension $\hypd(J(f))$ of the Julia set $J(f)$ is defined as the supremum of the Hausdorff dimensions of all forward invariant hyperbolic subsets of $J(f)$. It follows immediately from this definition that
$$
\hypd(J(f))\le \hd(J(f)).
$$

Let $c_{\feig}\approx -1.401155189$ be the Feigenbaum parameter, i.e., the limit of the period doubling bifurcations along the real axis, starting from the main cardioid of the Mandelbrot set. We prove the following:

\begin{theorem}\label{main_theorem_Feig}
	The hyperbolic dimension of the Julia set for the Feigenbaum polynomial $p_{c_{\feig}}$ satisfies $\hypd(J_{c_{\feig}})> 1.49781$.
\end{theorem}

We note that the best known upper bound on the Hausdorff dimension of the Julia set $J_{c_{\feig}}$ is $\hd(J_{c_{\feig}})<2$ (see \cite{DS-20}). Together with Theorem~C from~\cite{Avila_Lyubich_08} this implies that
$$
\hypd(J_{c_{\feig}}) = \hd(J_{c_{\feig}}).
$$

The proof of Theorem~\ref{main_theorem_Feig} involves computer assistance. All computations are done with explicit estimates of the accumulated errors, so the obtained lower bound is rigorous. 

Using the exact same machinery as in the proof of Theorem~\ref{main_theorem_Feig}, we obtain lower bounds on $\hypd(J_{c})$ for other parameters $c$. We illustrate this fact in Figure~\ref{fig:lower_bounds_plot}(a), where we plot lower bounds for 1000 evenly spaced parameters $-2\le c\le 2$.
Some parameters, corresponding to infinitely renormalizable real quadratic polynomials with stationary combinatorics, are of particular dynamical relevance; we single them out in the following theorem and compute the corresponding lower bounds with higher accuracy. 
Notice that any such parameters $c$ is uniquely determined by a unimodal permutation $s$ on the set $\{0,1,\ldots,n-1\}$, where $n$ is the smallest number, for which $p_c^n$ is hybrid equivalent to $p_c$. The permutation $s = [s_0,\dots,s_{n-1}]$ encodes the order of the points $0,p_c(0),\ldots,p_c^{n-1}(0)$ on the real line: $p_c^j(0)<p_c^k(0)$ for $0\leqslant j,k\leqslant n-1$ if and only if $s_j<s_k$.
\begin{theorem}\label{main_theorem_c}
The hyperbolic dimension of the Julia sets for the real-symmetric Feigenbaum polynomials $p_{c}$ with stationary combinatorics of length up to 6, satisfies the following lower bounds:
\begin{table}[h]
\begin{tabular}{m{0.3\textwidth} | m{0.25\textwidth} | m{0.3\textwidth}}
Unimodal permutation, describing the combinatorics & Approximation of the parameter $c$ & Lower bound on $\hypd(J_c)$ \\
\hline
$[1, 0, 5, 4, 3, 2]$ & $-1.9963832458$ & $1.03142410217842673$ \\
$[1, 0, 4, 3, 2]$ & $-1.9855395300$ & $1.07439037960430284$ \\
$[2, 0, 5, 4, 3, 1]$ & $-1.9668432010$ & $1.08899058048555264$ \\
$[1, 0, 3, 2]$ & $-1.9427043547$ & $1.15803646135900751$ \\
$[3, 0, 5, 4, 2, 1]$ & $-1.9075041928$ & $1.14436916492704910$ \\
$[2, 0, 4, 3, 1]$ & $-1.8622240226$ & $1.20002817922795790$ \\
$[1, 0, 2]$ & $-1.7864402555$ & $1.29622703845671050$ \\
$[2, 0, 5, 3, 1, 4]$ & $-1.7812168061$ & $1.32518996605940753$ \\
$[3, 0, 4, 2, 1]$ & $-1.6319266544$ & $1.33306905791978458$ \\
$[4, 0, 5, 2, 3, 1]$ & $-1.4831818301$ & $1.41584133336146056$
\end{tabular}
\end{table}
% -1.7798180758 & 1.32627542834379808 &  \\
\end{theorem}
The quantitative results of Theorem~\ref{main_theorem_Feig} and Theorem~\ref{main_theorem_c} are illustrated as red dots in Figure~\ref{fig:lower_bounds_plot}(b).
Conjecturally, all Julia sets from Theorem~\ref{main_theorem_c} have Hausdorff dimension strictly less than 2 (see~\cite{Dudko_20}), therefore, their Hausdorff and hyperbolic dimensions are expected to coincide.

\begin{figure}[h]
\begin{center}
\includegraphics[scale=0.55]{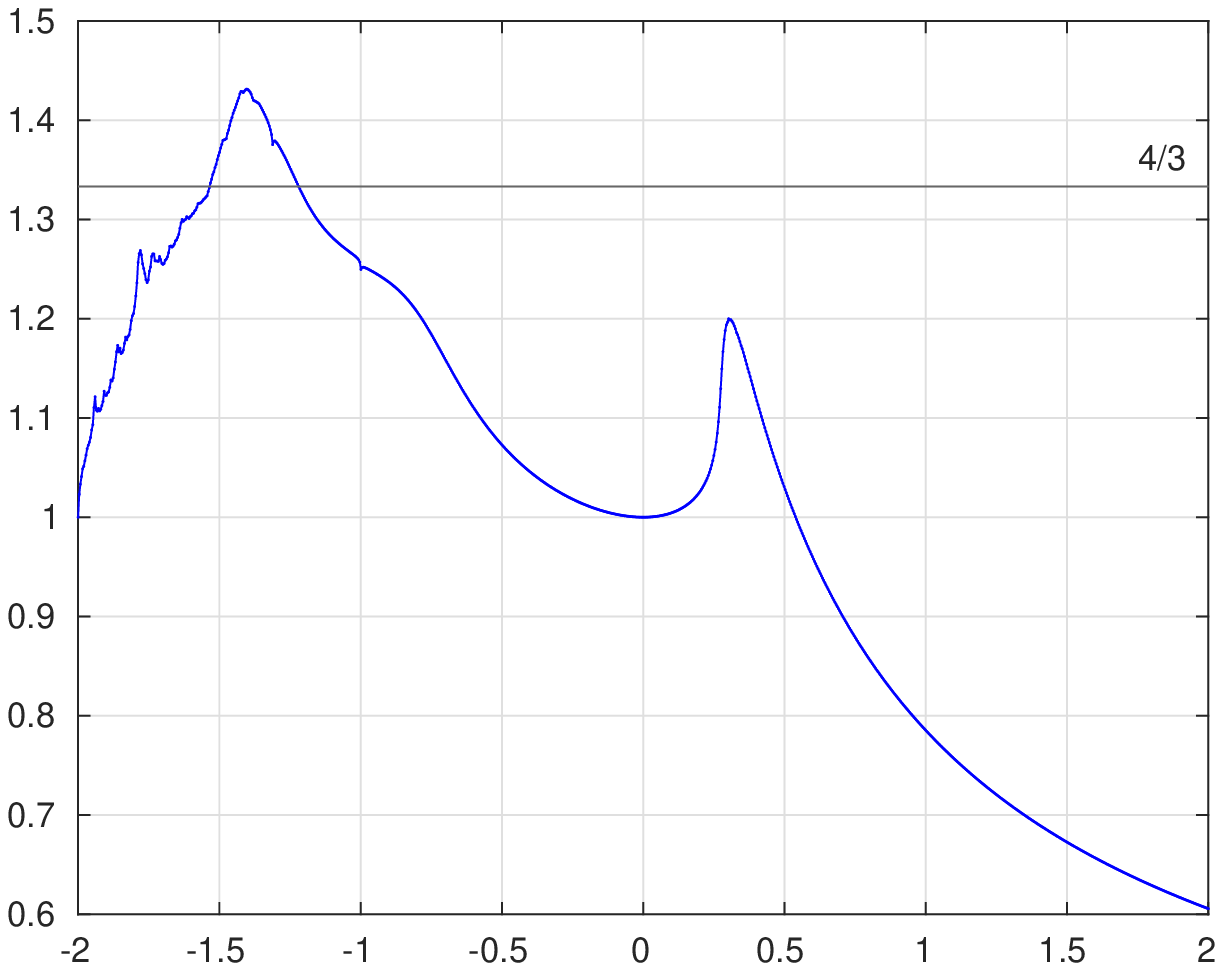}
\hspace*{-5mm}
\includegraphics[scale=0.55]{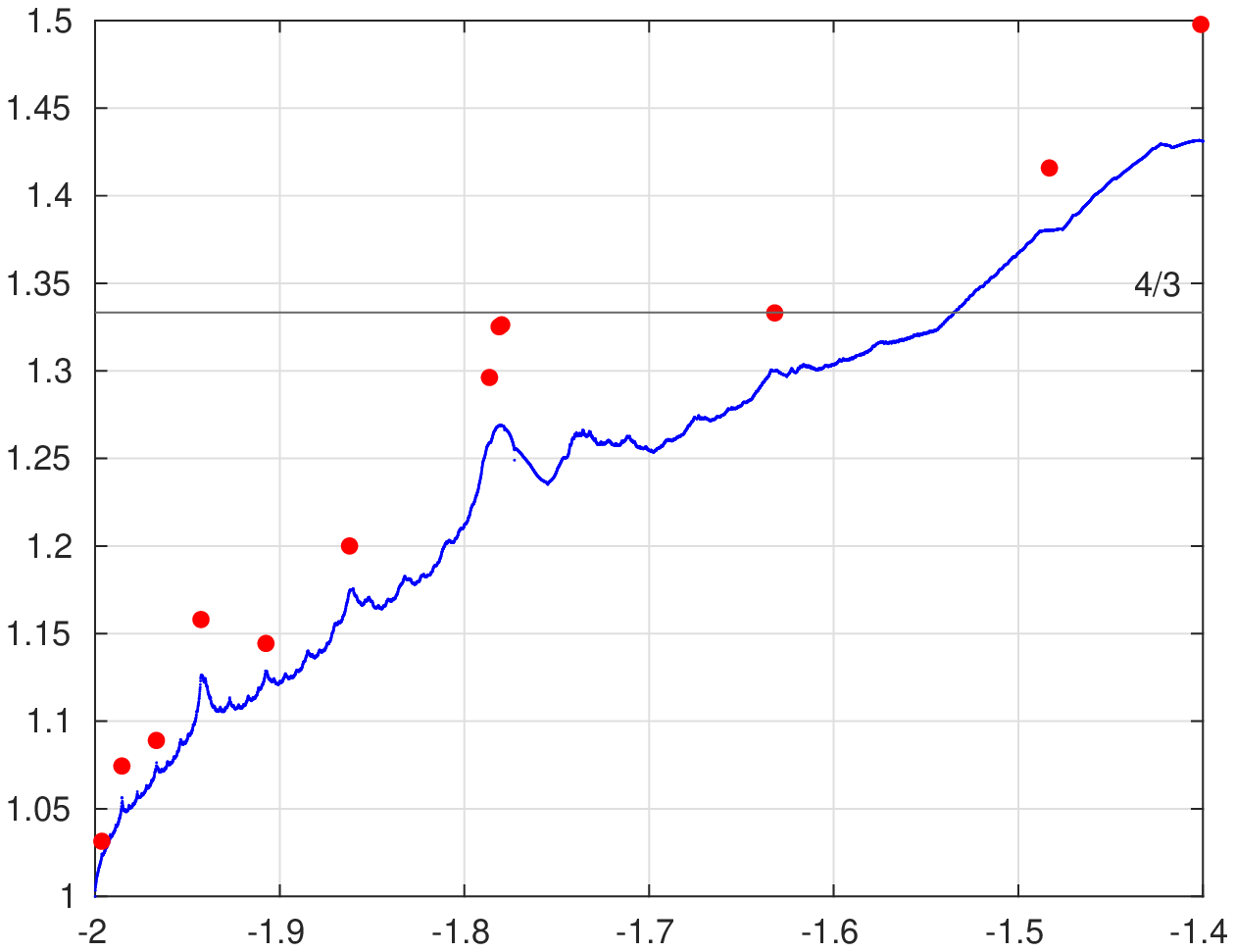}
%\subfigure{\raisebox{33.5mm}{\includegraphics[scale=0.55]{pictures/point_plot_depth_20.eps}}}
%\hspace*{-5mm}
%\subfigure{\includegraphics[scale=0.55]{pictures/point_plot_depth_20_28a_ig2.eps}}
\end{center}
\caption{Pointwise lower bounds on the hyperbolic dimension $\hypd(J_c)$. (a) 1000 evenly spaced parameters $c$ spanning the domain $[-2, 2]$. (b) The subdomain $[-2, -1.4]$ where red points correspond to Theorem~\ref{main_theorem_Feig} and Theorem~\ref{main_theorem_c}, and have been computed with higher accuracy than the other (blue) points.} %The red points are identified with their corresponding parameters on the Mandelbrot set.}
\label{fig:lower_bounds_plot}
\end{figure}

Finally, we apply our computational machinery to verify that for any real parabolic parameter $c\in(c_\feig, -3/4)$ we have $\hd(J_c)>4/3$. Combining this with a very recent result of Ludwik Jaksztas and Michel Zinsmeister (see \cite{JZ20}), we obtain the following theorem, conjectured in~\cite{JZ20}:

\begin{theorem}\label{main_theorem_smooth}
	For all (real) parameters $c\in (c_\feig, -3/4)$ the mapping $c\mapsto\hd(J_c)$ is continuously differentiable.
\end{theorem}

Refining and extending the computations required to prove Theorem~\ref{main_theorem_smooth}, we can construct a piecewise constant function whose graph forms a lower bound on the hyperbolic dimension over the entire parameter domain $[-2,2]$. An example of this is illustrated in Figure~\ref{fig:stairs_plot}(a), which constitutes a graphic representation of over 21000 small theorems (one lower bound for each subinterval). Note that the graph generally appears less smooth as the parameter $c$ decreases. Indeed, the Hausdorff dimension is expected to behave in a rather irregular manner for parameters $c\in\R$ outside of hyperbolic components, as indicated in Figure~\ref{fig:lower_bounds_plot} and Figure~\ref{fig:stairs_plot}(b). (See also~\cite{DGM-HD-20} for the study of the Hausdorff dimension $\hd(J_c)$ when $c\in\R$ is close to $-2$.)

\begin{figure}[h]
\begin{center}
\includegraphics[scale=0.55]{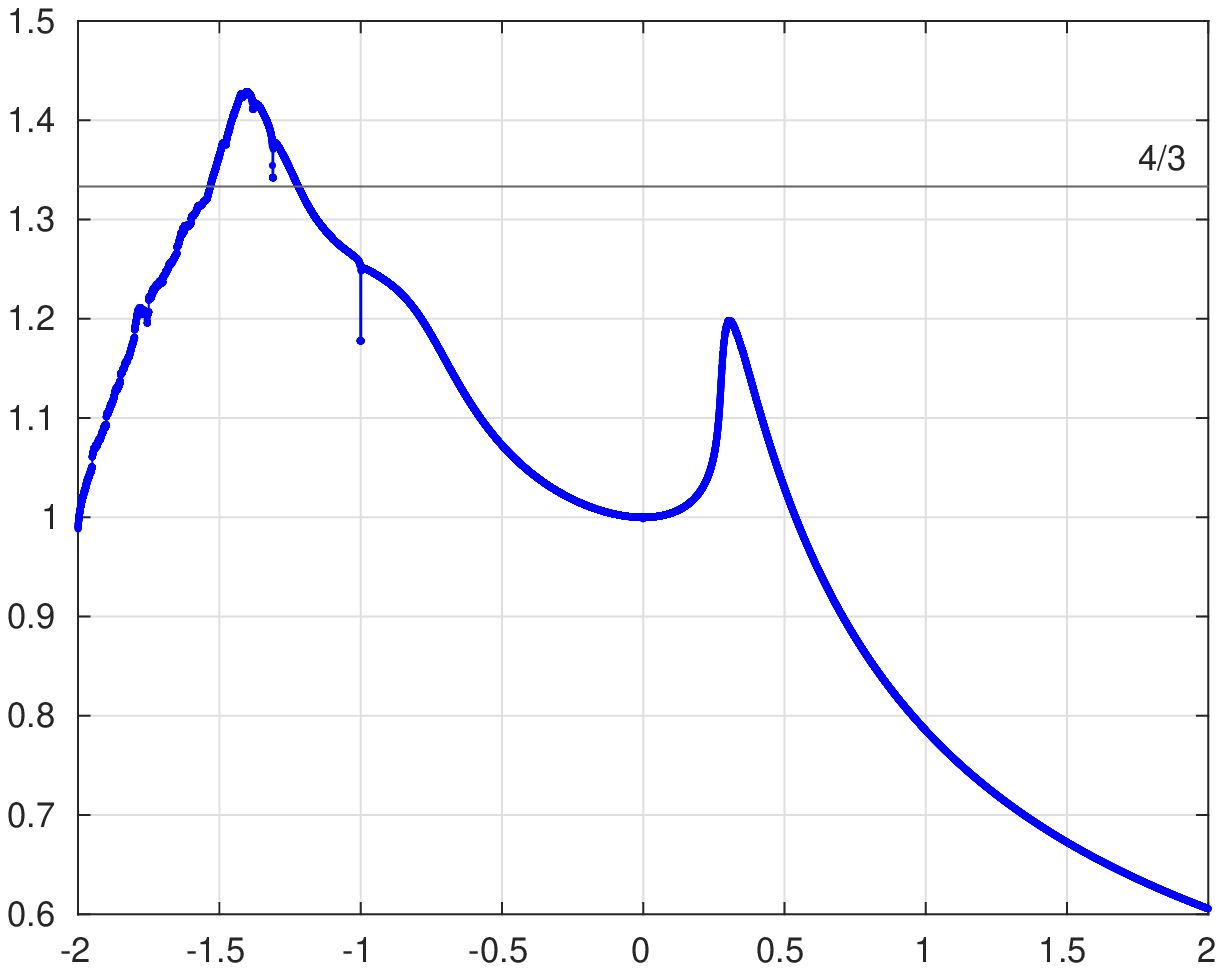}
\hspace*{-5mm}
\includegraphics[scale=0.55]{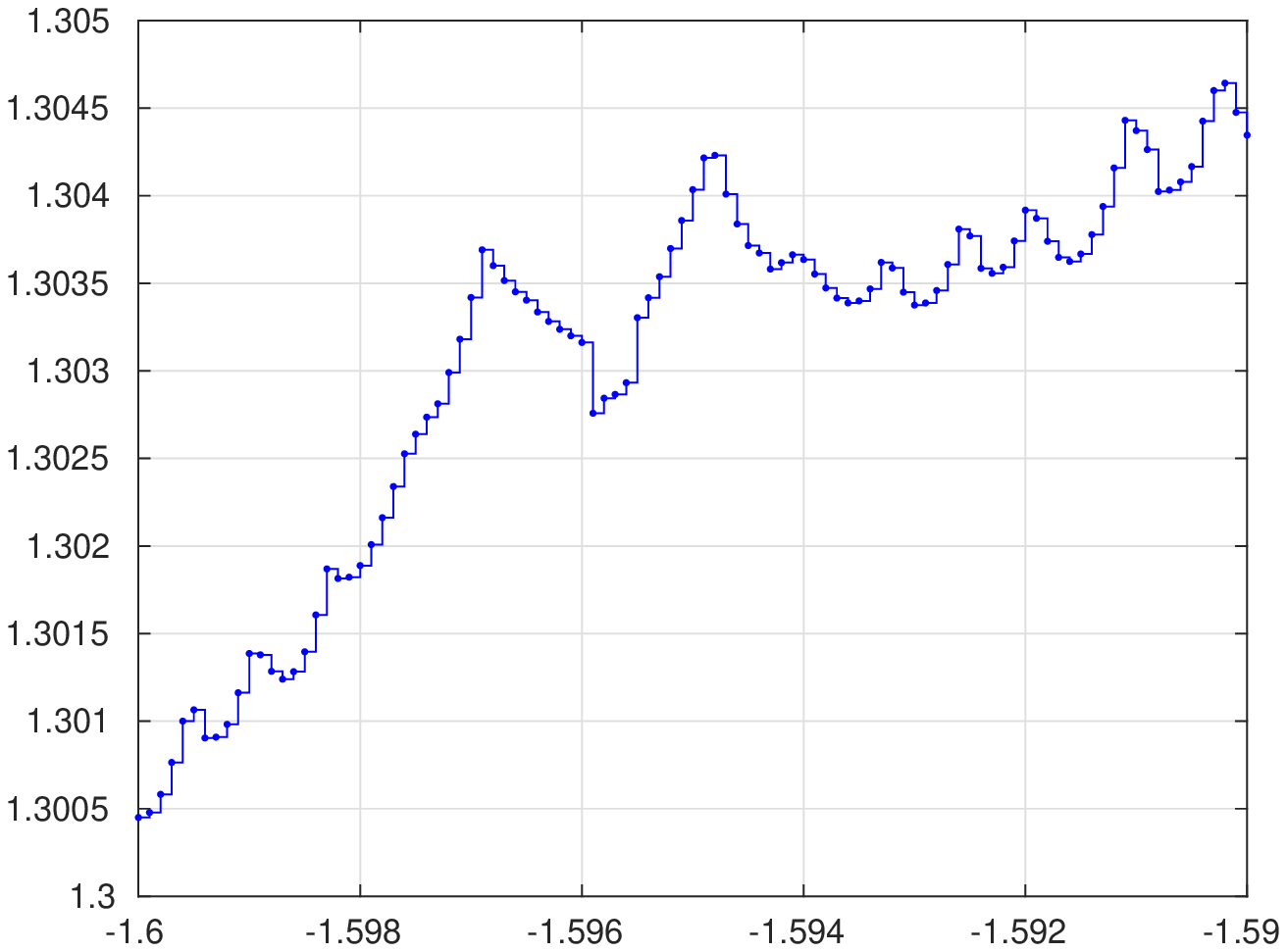}
\end{center}
\caption{A piecewise constant function bounding the hyperbolic dimension from below over the domain (a) $[-2, 2]$ where the right-most region displays ``regular'' behaviour. (b) A zoom into the region $[-1.6, -1.59]$ where some ``irregular'' behaviour is visible.}
\label{fig:stairs_plot}
\end{figure}

The paper is structured as follows:
in the first part of the paper (Sections~\ref{lower_bounds_sec} and~\ref{tree_pressure_section}) we present a general method for computation of lower bounds on the hyperbolic dimension of the Julia sets (see Theorem~\ref{main_inequality_theorem}). Our method is a rather straightforward modification of C.~McMullen's eigenvalue algorithm for expanding maps~\cite{McMullen-HDIII-98}. However, we show that our technique is applicable to a wider class of holomorphic maps, including non-expanding ones.
We also expect that the same method can be used for effective computation of lower bounds on the dimension of the limit sets of some Kleinian groups.

In the second part of this paper (Sections~\ref{Partitions_section} and~\ref{sec:the_computer_assisted_proof}) we apply our techniques to obtain rigorous lower bounds on the hyperbolic (and hence, also Hausdorff) dimension of the Julia sets that appear in Theorems~\ref{main_theorem_Feig},~\ref{main_theorem_c} and~\ref{main_theorem_smooth}, as well as in the discussion above. 
The involved computations are done using computer-assisted means with explicit estimates of the accumulated errors.

Last but not least, let us mention that another algorithm for estimating the Hausdorff dimension of the Julia sets was constructed by O.~Jenkinson and M.~Pollicott~\cite{JP-02}. However, this algorithm, as well as the original algorithm of C.~McMullen~\cite{McMullen-HDIII-98}, is designed for expanding maps, hence, is not directly applicable for establishing the results of the current paper. 
Estimates on the Hausdorff dimension of various other dynamically defined objects were obtained by many other authors (see for example~\cite{Pollicott_Vytnova_22}, and the references therein).

\textbf{Acknowledgments:} We would like to thank Feliks Przytycki for fruitful discussions. The research of I.G. was funded by the Deutsche Forschungsgemeinschaft (DFG, German Research Foundation) -- project number 455038303. A.D. was partially supported by National Science Centre, Poland, Grant OPUS21 "Holomorphic dynamics, fractals, thermodynamic formalism", 2021/41/B/ST1/00461.

\section{Lower bounds on the Hausdorff dimension}\label{lower_bounds_sec}

\subsection{Partitioned holomorphic dynamical systems}

We will consider a rather wide class of holomorphic dynamical systems, possessing certain kinds of Markov and transitivity properties.

\begin{definition}\label{hol_sys_def}
	A partitioned holomorphic dynamical system $\mathcal F$ on $\hat\C$ is a finite collection of holomorphic maps
	\begin{equation}\label{hol_sys_maps_eq}
	f_j\colon U_j\to \hat\C, \qquad\text{where } j=1,\dots,m,
	\end{equation}
	such that
	\begin{enumerate}
		\item the domains $U_j$ are pairwise disjoint
		open sets in $\hat\C$, and each $f_j$ is a proper branched (or unbranched) covering map of $U_j$ onto its image;
		\item (Markov property) for every pair of indices $i$, $j$, the image $f_i(U_i)$ either contains $U_j$, or is disjoint from it;
		\item (transitivity) for each domain $U_j$, there exists a finite sequence of maps $f_j=f_{i_0},f_{i_1},\dots,f_{i_k}$ from $\mathcal F$, such that the composition $f_{i_k}\circ\dots\circ f_{i_1}\circ f_{i_0}$ is defined on some open subset $U\subset U_j$ and maps $U$ surjectively onto the union $\cup_{i=1}^m U_i$.
	\end{enumerate}
The domains $U_j$ of the maps $f_j$ will be called the \textit{tiles} of $\mathcal F$.
\end{definition}

We note that, according to this definition, every rational or polynomial-like map is a partitioned holomorphic dynamical system that consists of a single map $f$. Furthermore, the restriction of a rational map on a Yoccoz puzzle of some fixed level can also be viewed as a partitioned holomorphic dynamical system with the number of maps being equal to the number of puzzle pieces. %\hlc{(Other example from Kleinian groups)}

Our definition of a partitioned holomorphic dynamical system can also be viewed as a generalization of a Markov partition for a conformal dynamical system from~\cite{McMullen-HDIII-98}. 
There is also a strong resemblance with the complex box mappings, introduced in~\cite{Kozlovski_vStrien_09}.

For a partitioned holomorphic dynamical system $\mathcal F$ that consists of the maps $f_1,\dots, f_m$ as in~(\ref{hol_sys_maps_eq}), let
$$
\mathcal U = \bigcup_{j=1}^m U_j \qquad\text{and}\qquad \mathcal F(\mathcal U) = \bigcup_{j=1}^m f_j(U_j)
$$
denote the union of the domains $U_j$ of the maps $f_j$ and respectively the union of the images $f_j(U_j)$. We say that $\mathcal U$ and $\mathcal F(\mathcal U)$ are the \textit{domain} and respectively, the \textit{range} of $\mathcal F$. Note that due to the transitivity property in Definition~\ref{hol_sys_def}, we have the inclusion $\mathcal U\subset\mathcal F(\mathcal U)$.

Since the domains $U_1,\dots,U_m$ are pairwise disjoint, we can view the partitioned holomorphic dynamical system $\mathcal F$ as a single map
$$
\mathcal F\colon\mathcal U\to\mathcal F(\mathcal U).
$$
Let $C(\mathcal F)$ be the set of all critical values of $\mathcal F$. That is,
$$
C(\mathcal F):= \{z\in\mathcal F(\mathcal U)\colon \exists w\in\mathcal F^{-1}(z),\text{ such that } \mathcal F'(w)=0\}.
$$
Every point $z\in\mathcal U$ has a finite or infinite forward orbit $\mathcal O(z)=\{z,\mathcal F(z),\mathcal F^{\circ 2}(z),\dots\}$ under the dynamics of $\mathcal F$. 
The postcritical set of $\mathcal F$,
$$
PC(\mathcal F) = \overline{\bigcup_{z\in C(\mathcal F)}\mathcal O(z)}
$$
is the closure of the union of all critical orbits of $\mathcal F$.

\begin{remark}\label{F_assumption_remark}
	\textbf{Assumption on $\mathcal F$:} In the remaining part of the paper we will always assume that $\mathcal F\colon\mathcal U\to\mathcal F(\mathcal U)$ is a partitioned holomorphic dynamical system consisting of maps~(\ref{hol_sys_maps_eq}) defined on the tiles $U_1,\dots,U_m$, and such that $\mathcal F(\mathcal U)\setminus PC(\mathcal F)$ is a nonempty open set. Due to the transitivity property of partitioned holomorphic dynamical systems, it follows that $U_j\setminus PC(\mathcal F)$ is also a nonempty open set, for each $j=1,\dots, m$. 
	Furthermore, for simplicity of exposition, we will also assume that the number of tiles $U_j$ is at least two.
\end{remark}

	\subsection{Geometric tree pressure and Hausdorff dimension}

For every positive integer $k\in\N$, a point $z\in \mathcal U\setminus PC(\mathcal F)$ and a real number $t>0$, consider the sum
$$
\Theta^k(z,t,\mathcal F) = \sum_{w\in\mathcal U\colon z=\mathcal F^{\circ k}(w)} |(\mathcal F^{\circ k})'(w)|^{-t},
$$
where the summation is taken over all $w\in\mathcal U$, such that $\mathcal F^{\circ k}(w)=z$. The derivatives here are considered with respect to the spherical metric on $\hat{\C}$. The \textit{geometric tree pressure} $P(z,t,\mathcal F)$ is then defined as
$$
P(z,t,\mathcal F) = \limsup_{k\to\infty}\frac{1}{k}\log \Theta^k(z,t,\mathcal F).
$$
For rational maps of the Riemann sphere, the geometric tree pressure was studied in~\cite{Prz_99}. Various equivalent definitions were also considered in~\cite{PRS_2004} and~\cite{Prz22}.

It is not difficult to show that if for every $j=1,\dots,m$, the set $U_j\setminus PC(\mathcal F)$ is path connected, then the pressure $P(z,t,\mathcal F)$ does not depend on the point $z\in \mathcal U\setminus PC(\mathcal F)$. Indeed, under these conditions, by the Koebe Distortion Theorem (see Lemmas~\ref{Koebe1} and~\ref{Koebe2}), the pressure $P(z,t,\mathcal F)$ does not depend on the choice of $z$ within the same domain $U_j$. Then, the transitivity property of Definition~\ref{hol_sys_def} implies that $P(z,t,\mathcal F)$ is the same for any choice of $z$ in $\mathcal U\setminus PC(\mathcal F)$.

In case if $\mathcal F$ is a rational map on the Riemann sphere (without any assumptions on the geometry of the postcritical set), it was shown in~\cite{Prz_99} that $P(z,t,\mathcal F)$ is independent of the choice of $z\in\hat\C\setminus PC(\mathcal F)$. %\hlc{Remove: , provided that it avoids a certain small exceptional set of Hausdorff dimension zero.}
We do not attempt here to prove a similar statement for all partitioned holomorphic dynamical systems $\mathcal F$. Instead, it is sufficient for us to consider the pressure function
$$
P(t,\mathcal F) = \inf_{z\in \mathcal U\setminus PC(\mathcal F)} P(z,t,\mathcal F),
$$
and define the Poincar\'e exponent $t_\mathcal F$ as follows: %critical parameter
$$
t_\mathcal F = \inf\{t>0\colon P(t,\mathcal F)\le 0\}.
$$

\begin{lemma}\label{pressure_at_2_lemma}
	The inequality $P(2,\mathcal F)\le 0$ holds. Hence, the set $\{t>0\colon P(t,\mathcal F)\le 0\}$ is nonempty and the Poincar\'e exponent %critical parameter
	$t_\mathcal F$ is well defined.	
\end{lemma}
\begin{proof}
	Let $z\in\mathcal U\setminus PC(\mathcal F)$ be an arbitrary point, and choose a sufficiently small spherical disk $D\subset \mathcal U\setminus PC(\mathcal F)$, centered at $z$. Then, according to the Koebe Distortion Theorem (see Lemmas~\ref{Koebe1} and~\ref{Koebe2}), for any $k\in\N$ and any $w\in\mathcal F^{-k}(z)$, the term $|(\mathcal F^{\circ k})'(w)|^{-2}$ is commensurable with the spherical area of the connected component of $\mathcal F^{-k}(D)$, containing $w$. Summing up over all such connected components, we conclude that $\Theta^k(z,2,\mathcal F)$ is uniformly bounded in $k$. Thus, $P(z,2,\mathcal F)\le 0$.
\end{proof}

\begin{remark}\label{metric_remark}
	If $\infty\not\in\mathcal U$ and $\mathcal U$ is a bounded subset of $\C$, then all computations in this, as well as in the following sections can be done in the Euclidean metric instead of the spherical one. Since these two metrics are equivalent on $\mathcal U$, this does not change the pressure function $P(t,\mathcal F)$. All further proofs repeat verbatim.
\end{remark}

Let us mention an important relation between the Poincar\'e exponent $t_\mathcal F$ and the dimension of the Julia set of $\mathcal F$, in case if $\mathcal F$ is a rational map. More precisely, if $f\colon\hat{\C}\to\hat{\C}$ is a rational map of degree $d\ge 2$, let $J(f)\subset\hat{\C}$ denote its Julia set. We recall that the hyperbolic dimension $\hypd(J(f))$ of the Julia set $J(f)$ is defined as the supremum of the Hausdorff dimensions of all forward invariant hyperbolic subsets of $J(f)$. It follows immediately from this definition that
$$
\hypd(J(f))\le \hd(J(f)).
$$
The following fundamental result is a generalized version of Bowen's formula (see~\cite{Bowen_79,Prz_99}).

\begin{theorem}\label{Bowen_formula_theorem}
Let $\mathcal F\colon\mathcal U\to\mathcal F(\mathcal U)$ be a partitioned holomorphic dynamical system that is a restriction of a rational map $f\colon\hat{\C}\to\hat{\C}$ of degree $d\ge 2$. Assume that $\mathcal F$ viewed as a single map has the same topological degree $d$ as the rational map $f$. Then
$$
t_\mathcal F = \hypd(J(f)).
$$
\end{theorem}

We note that even though hyperbolic and Hausdorff dimensions of a Julia set are not necessarily equal in general, they are known to be equal for certain classes of rational maps $f$ including hyperbolic ones, Collet-Eckmann rational maps~\cite{Prz_99}, and Feigenbaum maps whose Julia sets have zero Lebesgue measure~\cite{Avila_Lyubich_08}.

	\subsection{Construction of a lower bound on $t_\mathcal F$}

A key difficulty one can face when working with a Poincar\'e exponent $t_\mathcal F$ is that they are often difficult to compute. In this subsection we provide a quantity $\delta_\mathcal F$ that is usually much more tangible for direct rigorous computations than the Poincar\'e exponent 
itself. In Section \ref{tree_pressure_section} we prove that $\delta_\mathcal F \leqslant t_\mathcal F$. Afterwards, in Section \ref{Partitions_section} we discuss application of the latter result to the real quadratic family. Finally, in Section \ref{sec:the_computer_assisted_proof} we provide details on computer assisted proofs of Theorems \ref{main_theorem_Feig}, \ref{main_theorem_c} and \ref{main_theorem_smooth}. 

Assume, $\mathcal F$ is a partitioned holomorphic dynamical system, consisting of the maps~(\ref{hol_sys_maps_eq}). For each index $j=1,\dots,m$, let $d_j$ be the topological degree of the (branched) covering map $f_j$. Now, for any $t\ge 0$, we introduce the so called, McMullen's matrix $M(t) = M(\mathcal F,t)$ whose $(i,j)$-th element $m_{ij}$ is defined as
\begin{equation*}
m_{ij} = m_{\mathcal F,t}(i,j) = \begin{cases}
d_i\left(\sup_{z\in U_i \cap f_i^{-1}(U_j)} |f_i'(z)| \right)^{-t} & \text{if } U_i \cap f_i^{-1}(U_j) \neq\varnothing\\
0, & \text{otherwise},
\end{cases}
\end{equation*}
where the supremum above is taken over all $z\in U_i$, such that $f_i(z)\in U_j$, and the derivatives are computed in the spherical metric (see also Remark~\ref{metric_remark}). 

For a square matrix $A$, let $\varrho(A)$ denote the spectral radius of $A$. That is, $\varrho(A)$ is the maximum of the absolute values of the eigenvalues of $A$. 
For a partitioned holomorphic dynamical system $\mathcal F$, we define
\begin{equation}\label{delta_F_eq}
\delta_{\mathcal F} = \inf\{t\ge 0\colon \varrho(M(\mathcal F,t))= 1\}.	
\end{equation}

Finally, we can state the lower bound on the Poincar\'e exponent 
$t_\mathcal F$:

\begin{theorem}\label{main_inequality_theorem}
	For any partitioned holomorphic dynamical system $\mathcal F$, satisfying the assumptions from Remark~\ref{F_assumption_remark}, the number $\delta_{\mathcal F}$ is well defined, and $\delta_{\mathcal F} \le t_\mathcal F$.
\end{theorem}

We finish this subsection with several practical remarks.

\begin{remark}\label{primitive_remark}
First, we note that any McMullen's matrix $M$ always has nonnegative entries. Furthermore, due to the transitivity condition in Definition~\ref{hol_sys_def}, any McMullen's matrix $M$ is primitive. That is, there exists $k\in\N$, such that the $k$-th power, $M^k$ has only positive entries. The spectral radius of a primitive matrix can be rigorously estimated both from above and from below by a version of the Perron-Frobenius Theorem, obtained by Collatz in~\cite{Co42} (see Theorem~\ref{thm:collatz}).
\end{remark}

\begin{remark}
The following general fact turns out to be helpful in obtaining tight rigorous bounds on $\delta_{\mathcal F}$:
\begin{lemma}\label{concavity_lemma}
	Let $M=\{m_{i,j}\}_{1\leqslant i,j\leqslant n}$ be a $n\times n$ primitive matrix and $M(\delta)=\{m_{i,j}^\delta\}_{1\leqslant i,j\leqslant n}$, where $\delta \in\mathbb R$. Then the spectral radius $\varrho(M(\delta))$ is a concave up function on $\mathbb R$.
\end{lemma}
\noindent The proof of Lemma~\ref{concavity_lemma} is provided in the Appendix.

In particular, for McMullen's matrices Lemma~\ref{concavity_lemma} implies that the function $\delta\mapsto \varrho(M(\mathcal F,\delta))$ is concave up on $(0,\infty)$. Hence, if we know that for some initial $\delta_1< \delta_2$, we have $\varrho(M(\mathcal F,\delta_1))>1> \varrho(M(\mathcal F,\delta_2))$, then rigorous tight bounds on $\delta_{\mathcal F}$ can be obtained by a standard bisection search.
\end{remark}

\begin{remark}
Finally, the lower bound $\delta_{\mathcal F}$ on $t_\mathcal F$ can be improved by considering a refinement of~$\mathcal F$.
\begin{definition}
	Let $\mathcal F$ be a partitioned holomorphic dynamical system, consisting of the maps~(\ref{hol_sys_maps_eq}). Another partitioned holomorphic dynamical system $\hat{\mathcal F}\colon\hat{\mathcal U}\to\hat{\mathcal F}(\hat{\mathcal U})$, consisting of the maps
	\begin{equation*} %\label{hol_sys_maps_eq}
	\hat f_j\colon \hat U_j\to \hat\C, \qquad\text{where } j=1,\dots,\hat m,
	\end{equation*}
	is a \emph{refinement} of $\mathcal F$, if each tile $\hat U_j$ is contained in some tile $U_i$, the corresponding map $\hat f_j$ is the restriction of $f_i$ to $\hat U_j$, and for each point $z\in\hat{\mathcal F}(\hat{\mathcal U})$, the sets $\hat{\mathcal F}^{-1}(z)$ and $\mathcal F^{-1}(z)$ coincide.	
\end{definition}

It is obvious from this definition that if $\hat{\mathcal F}$ is a refinement of $\mathcal F$, then $t_\mathcal F = t_{\hat{\mathcal F}}$. On the other hand, we show in Remark~\ref{Refinement_delta_remark} that $\delta_{\mathcal F}\le \delta_{\hat{\mathcal F}}$. Hence, one can hope to get a better lower bound on the Poincar\'e exponent %critical parameter
$t_\mathcal F$ by applying Theorem~\ref{main_inequality_theorem} to a refinement of $\mathcal F$, rather than to the system $\mathcal F$ itself. We implement this approach for real quadratic polynomials in Section~\ref{sec:the_computer_assisted_proof}.
\end{remark}

\section{Tree pressure over pseudo-orbits}\label{tree_pressure_section}

In this section we give a proof of Theorem~\ref{main_inequality_theorem}.

Assume, $k\in\N$ and $z\in \mathcal U$ are such that $z_i=\mathcal F^i(z)$ is defined for all $i=0,1,\dots,k$ and $z_k\in\mathcal U$. Let $j_0,j_1,\dots,j_k\in [1,m]$ be the sequence of indices, such that for each $i=0,\dots,k$, we have $z_i\in U_{j_i}$. We will call this sequence, \textit{the itinerary of $z$ of length $k+1$}. Then for each $i=0,\dots, k-1$, the inclusion $z_i\in U_{j_i}\cap \mathcal F^{-1}(U_{j_{i+1}})$ holds, and one can define the expression
\begin{equation}\label{Q_eq}
Q(z, k):= \prod_{i=0}^{k-1} \left(\sup_{y\in U_{j_i} \cap f_{j_i}^{-1}(U_{j_{i+1}})} |f_i'(y)| \right).
\end{equation}

It follows immediately from the Chain Rule that the following inequality holds:

\begin{equation}\label{F_Q_ineq}
|(\mathcal F^{\circ k})'(z)| \le Q(z,k).
\end{equation}

For every $z\in \mathcal U\setminus PC(\mathcal F)$ and $t>0$, consider the expression
$$
\theta^k(z,t,\mathcal F) = \sum_{w\in\mathcal U\colon z=\mathcal F^{\circ k}(w)} (Q(w,k))^{-t},
$$
where the summation is taken over all $w\in\mathcal U$, such that $\mathcal F^{\circ k}(w)=z$.
It follows immediately from~(\ref{F_Q_ineq}) that the inequality
\begin{equation}\label{Theta_ineq}
\theta^k(z,t,\mathcal F) \le \Theta^k(z,t,\mathcal F)
\end{equation}
holds for every $z\in \mathcal U\setminus PC(\mathcal F)$, $k\ge 1$ and $t>0$.
One can also observe that the expressions $\theta^k(z,t,\mathcal F)$ are independent from $z$ within the same tile $U_j$. Indeed, according to the Markov property of partitioned holomorphic dynamical systems, if $z,\tilde z\in U_j\setminus PC(\mathcal F)$, then for all $k\in\N$ there is a bijection between the finite sets $\mathcal F^{-k}(z)$ and $\mathcal F^{-k}(\tilde z)$, where the bijection is taking each point $w$ of $\mathcal F^{-k}(z)$ into a point $\tilde w$ of $\mathcal F^{-k}(\tilde z)$ with the same itinerary of length $k+1$. The latter implies that $Q(w, k) = Q(\tilde w, k)$, and the claim follows.

For any $z\in\mathcal U\setminus PC(\mathcal F)$ and $t>0$, consider the following version of the pressure function:
$$
p(z,t,\mathcal F) = \limsup_{k\to\infty}\frac{1}{k}\log \theta^k(z,t,\mathcal F).
$$
In particular,~(\ref{Theta_ineq}) implies that
\begin{equation}\label{pressure_ineq}
p(z,t,\mathcal F) \le P(z,t,\mathcal F).
\end{equation}
The next lemma states the relation between the values of $\theta^k(z,t,\mathcal F)$, $p(z,t,\mathcal F)$ and McMullen's matrices. For each index $j=1,\dots,m$, let $\mathbf e_j\in\R^m$ be the column vector whose $j$-th coordinate is $1$ and all other coordinates are $0$. Let also $\mathbf 1\in\R^m$ be the column vector, whose coordinates are all equal to $1$.

\begin{lemma}\label{Mktheta_lemma}
	Let $t>0$ be a real number and let $M=M(\mathcal F,t)$ be the corresponding McMullen's matrix. Then, for any pair of integers $k\in\N$, $j\in [1,m]$ and any $z\in U_j\setminus PC(\mathcal F)$, the following holds:
	$$
	\theta^k(z,t,\mathcal F) = \mathbf 1^T M^k \mathbf e_j.
	$$
	As a consequence, we have
	$$
	p(z,t,\mathcal F) = \log\varrho(M(\mathcal F,t)).
	$$
	In particular, $p(z,t,\mathcal F)$ is independent of $z\in\mathcal U\setminus PC(\mathcal F)$.
\end{lemma}
\begin{proof}
	It follows immediately by induction on $k$ that the $(i,j)$-th entry $M^k_{ij}$ of the matrix $M^k$ is equal to
	$$
	M^k_{ij} = \sum_{w\in U_i\colon z=\mathcal F^{\circ k}(w)} (Q(w,k))^{-t},
	$$
	where the summation is taken over all $w\in U_i$, such that $\mathcal F^{\circ k}(w)=z$. (If there are no such points $w$, then the sum is considered to be equal to zero.) Then $\theta^k(z,t,\mathcal F)$ is equal to $\mathbf 1^T M^k \mathbf e_j$, which is the sum of all elements in the $j$-th column of $M^k$.	
	
	Since the matrix $M = M(\mathcal F,t)$ is primitive (see Remark~\ref{primitive_remark}), the second part of the lemma follows from the version of the Perron-Frobenius Theorem, which states that $M^k/(\varrho(M))^k$ converges to a matrix, whose columns are positive multiples of the leading eigenvector.
\end{proof}

\begin{proof}[Proof of Theorem~\ref{main_inequality_theorem}]
	From Lemma~\ref{Mktheta_lemma} and~(\ref{pressure_ineq}) it follows that
	$$
	\log \varrho(M(\mathcal F,t)) = p(z,t,\mathcal F)\le P(z,t,\mathcal F),
	$$
	for any $z\in\mathcal U\setminus PC(\mathcal F)$ and $t>0$. Now, passing to the infimum over $z\in\mathcal U\setminus PC(\mathcal F)$, we obtain
	
	\begin{equation}\label{pressure_spec_rad_ineq}
	\log \varrho(M(\mathcal F,t)) \le P(t,\mathcal F).
	\end{equation}
	Observe that $\log \varrho(M(\mathcal F,0)) = p(z,0,\mathcal F) \ge 0$ and $\log \varrho(M(\mathcal F,2)) \le P(2,\mathcal F)\le 0$ (see Lemma~\ref{pressure_at_2_lemma}). Since the function $t\mapsto \log \varrho(M(\mathcal F,t))$ is continuous, the above inequalities imply that this function attains a zero on the interval $[0,2]$, hence, the set $\{t\ge 0\colon \varrho(M(\mathcal F,t))= 1\}$ from~(\ref{delta_F_eq}) is nonempty and the number $\delta_{\mathcal F}$ is well defined.
	The inequality $\delta_{\mathcal F}\le t_{\mathcal F}$ follows from~(\ref{pressure_spec_rad_ineq}).
\end{proof}

We conclude this section with several observations.
\begin{remark}\label{Refinement_delta_remark}
	If $\hat{\mathcal F}\colon\hat{\mathcal U}\to\hat{\mathcal F}(\hat{\mathcal U})$ is a refinement of $\mathcal F$, then for any $z\in\hat{\mathcal U}\setminus PC(\hat{\mathcal F})$ and $t>0$, we have
	$$
	p(z,t,\mathcal F)\le p(z,t,\hat{\mathcal F}),
	$$
	since the suprema in~(\ref{Q_eq}) for a refinement $\hat{\mathcal F}$ are taken over smaller domains compared to the case of $\mathcal F$. Thus, $\varrho(M(\mathcal F,t)) \le \varrho(M(\hat{\mathcal F},t))$ and in particular,
	$$\delta_{\mathcal F}\le \delta_{\hat{\mathcal F}}.$$
\end{remark}

\begin{remark}\label{Convergence_delta_remark}
Let $\mathcal F_k$ be a sequence of partitioned holomorphic dynamical systems such that $\mathcal F_{k+1}$ is a refinement of $\mathcal F_k$ for every $k\in\mathbb N$. Then for any $z\in\hat{\mathcal U}\setminus PC(\hat{\mathcal F})$ and $t>0$ the sequence $p(z,t,\mathcal F_k)$ is increasing and therefore converges.
Assume that $\mathcal F_1$ (and hence, all other $\mathcal F_k$) is a restriction of a rational function $f$ and the topological degree of each $\mathcal F_k\colon \mathcal U_k\to\mathcal F_k(\mathcal U_k)$ coincides with $\deg f$. Assume that $\diam(\mathcal F_k)$, defined as the maximum of the diameters of all tiles of $\mathcal F_k$, converges to zero when $k\to\infty$.
In \cite{Prz22} F.~Przytycki showed that under the above conditions, $\lim p(z,t,\mathcal F_k)=P(t,f)$, where $P(t,f)$ is the geometric tree pressure of a rational function $f$. Taking into account that the sequence $p(z,t,\mathcal F_k)$ is increasing in $k$, we obtain that $\delta_{\mathcal F_k}\to t_f = \hypd(J(f))$, as $k\to\infty$. 
\end{remark}

\section{Dynamical partitions for real quadratic polynomials}\label{Partitions_section}

Starting from this section we restrict our consideration to the real quadratic polynomials $p_c(z) = z^2+c$, where $c\in [-2,2]$.

Let $D_r(z)\subset\C$ denote the open disc of radius $r\ge 0$ centered at $z\in\C$, and let $U_1, U_2\subset\C$ be the two connected components (half-disks) of the set $D_2(0)\setminus [-2,2]$. For any fixed $c\in [-2,2]$, consider the maps $f_1\colon U_1\to\C$ and $f_2\colon U_2\to\C$ defined as the restrictions of $p_c$ to the corresponding domains $U_1$ and $U_2$. It is easy to check that the maps $f_1$ and $f_2$ together form a partitioned holomorphic dynamical system $\mathcal F_c$ as in Definition~\ref{hol_sys_def}. (Note that the latter is not true for non-real complex parameters $c$ due to failure of the Markov property in Definition~\ref{hol_sys_def}.)

A sequence of refinements of $\mathcal F_c$ can be constructed in a standard way: a refinement $\mathcal F_c^{(k)}$ of level $k = 2, 3,\dots$ consists of $2^k$ maps, each of them being a restriction of the polynomial $p_c$ to a corresponding connected component of the set $p_c^{1-k}(U_1\cup U_2)$. We refer to these connected components as tiles of level $k$ and denote them by $P_j^{(k)}(c)$, where $j=1,2,\dots, 2^k$. The tiles are enumerated so that each tile $P_j^{(k)}(c)$ is squeezed between the external rays of angles $(j-1)/2^k$ and $j/2^k$. For each $k$ as above, the set of all tiles of level $k$ is denoted by $\mathcal P^{(k)}(c)$. We will often suppress the parameter $c$ in the notation, writing $P_j^{(k)}$ and $\mathcal P^{(k)}$ in place of $P_j^{(k)}(c)$ and $\mathcal P^{(k)}(c)$ respectively, provided that this does not cause any ambiguity.

It follows from Theorem~\ref{Bowen_formula_theorem} that for any $k$, the Poincar\'e exponents %critical parameter
$t_{\mathcal F_c^{(k)}}$ coincides with the hyperbolic dimension of the Julia set $J_c$, hence, due to Theorem~\ref{main_inequality_theorem}, we have
$$
\delta_{\mathcal F_c^{(k)}}\le \hypd(J_c)\le \hd(J_c),
$$
for any $k\ge 2$. According to Remark~\ref{Refinement_delta_remark}, the difference $\hypd(J_c)-\delta_{\mathcal F_c^{(k)}}$ decreases monotonically as $k$ increases, so our strategy for obtaining lower bounds on $\hypd(J_c)$ consists of estimating $\delta_{\mathcal F_c^{(k)}}$ from below for sufficiently large values of $k$.

The McMullen matrices, corresponding to the refinements $\mathcal F_c^{(k)}$ will be denoted by
$$
M_c^{(k)}(t):= M(\mathcal F_c^{(k)},t)
$$
with the convention that $M_c^{(k)} = M_c^{(k)}(1)$. We will often suppress the indices and write $M^{(k)}(t)$ or $M(t)$ instead of $M_c^{(k)}(t)$, and $M^{(k)}$ or simply $M$ instead of $M_c^{(k)}$, when this does not cause any confusion.

%Refinement, introduce notation for tiles, partitions, McMullen matrices...

We note that due to local connectivity of the Julia sets for real quadratic polynomials, it follows that $\diam(\mathcal F_c^{(k)})\to 0$ as $k\to\infty$. Hence, the result of F.~Przytycki~\cite{Prz22}, discussed in Remark~\ref{Convergence_delta_remark}, implies that
$$
\lim_{k\to\infty}\delta_{\mathcal F_c^{(k)}} = \hypd(J_c),
$$
for all $c\in[-2,2]$. Unfortunately, we don't know a good bound on the rate of convergence of $\delta_{\mathcal F_c^{(k)}}$, when the critical point is recurrent (e.g., when $p_c$ is an infinitely renormalizable quadratic polynomial).

	\section{The computer-assisted proof}
\label{sec:the_computer_assisted_proof}

In this section we will outline the general strategy employed for the computer-assisted part of the proof. We will also present the key steps of the proof, together with some details that are specific to the problem at hand. We begin by recalling the main results, summarizing our conclusions.

\begin{theorem}\label{thm:lower_bound_on_hd}
Define $c_\star = -1.4011551890$. Then for any parameter $c\in\R$ satisfying $|c - c_\star|\le 10^{-10}$, we have the lower bound
\begin{equation*}
1.49781 < \hypd(J_c).
\end{equation*}
\end{theorem}

We remark that the Feigenbaum parameter $c_{\feig}$ satisfies the bound
$|c_{\feig} - c_\star|\le 10^{-10}$, so the theorem applies to the Feigenbaum map $p_{\feig}(z) = z^2 + c_{\feig}$. Thus, Theorem~\ref{thm:lower_bound_on_hd} immediately implies Theorem~\ref{main_theorem_Feig}. That is, the Hausdorff and hyperbolic dimensions of the Julia set for the Feigenbaum map satisfy
$$
1.49781 < \hypd(J_{c_{\feig}}) = \hd(J_{c_{\feig}}).
$$
(The equality $\hypd(J_{c_{\feig}}) = \hd(J_{c_{\feig}})$ follows from~\cite{Avila_Lyubich_08} and~\cite{DS-20}.)

By using our computational machinery developed to prove Theorem~\ref{thm:lower_bound_on_hd}, we can study other parameters in the same manner. Interesting candidates are the periodic Feigenbaum parameters of which we explore a few below:

\begin{theorem}\label{thm:lower_bound_on_hd_c}
Given a parameter $c_\star$ from Table~\ref{lo_bounds_table}, %the table below, 
for any parameter $c\in\R$ satisfying $|c - c_\star|\le 10^{-10}$, the hyperbolic dimension of its Julia set $\hypd(J_c)$ satisfies the corresponding tabulated lower bound.
\end{theorem}

\begin{table}[h]
\begin{tabular}{ll}
\hspace*{2mm} parameter $c_\star$ & \hspace*{5mm} lower bound  \\
\hline
$-1.9963832458$ & $1.03142410217842673$  \\
$-1.9855395300$ & $1.07439037960430284$  \\
$-1.9668432010$ & $1.08899058048555264$  \\
$-1.9427043547$ & $1.15803646135900751$  \\
$-1.9075041928$ & $1.14436916492704910$  \\
$-1.8622240226$ & $1.20002817922795790$  \\
$-1.7864402555$ & $1.29622703845671050$  \\
$-1.7812168061$ & $1.32518996605940753$  \\
$-1.6319266544$ & $1.33306905791978458$  \\
$-1.4831818301$ & $1.41584133336146056$
\end{tabular}
\vspace{5mm}
\caption{Lower bounds for the hyperbolic dimension of the Julia sets}\label{lo_bounds_table}
\end{table}

This immediately implies Theorem~\ref{main_theorem_c}, where we also have added some information regarding the dynamics of the critical orbits. The computational scheme for rigorous numerical approximations of the parameters $c_\star$ from the table is discussed in the Appendix.

Feeding even wider sets of parameters to our code, we can compute a uniform lower bound on the Hausdorff dimension of the associated Julia sets:
\begin{theorem}\label{thm:uniform_bound_on_hd}
For any (real) parabolic parameter $\hat c \in (c_\feig, -3/4)$ we have $\hypd(J_{\hat c}) > 4/3$.
\end{theorem}

Combining Theorem~\ref{thm:uniform_bound_on_hd} with a very recent result of Ludwik Jaksztas and Michel Zinsmeister (see \cite{JZ20}), we obtain the result of Theorem~\ref{main_theorem_smooth}, which says that for all (real) parameters $c\in (c_\feig, -3/4)$ the mapping $c\mapsto\hd(J_c)$ is continuously differentiable.

\subsection{Constructing the McMullen matrices}
\label{subsec:constructing_the_mcmullen_matrices}

A key step of our proof is to be able to compute the spectral radii of certain variants of matrices $M_c^{(k)} = M^{(k)}$ $(k=2,3,\dots)$, here called McMullen matrices.
The parameter $c$ is assumed to be in the interval $c\in[-2,2]$.
The positive integer $k$ corresponds to the encoding depth used when constructing the cover of the Julia set, see Section~\ref{Partitions_section}. From this point of view, $k$ is the length of the bit strings used to encode the $2^k$ tiles at level $k$. As a consequence, the matrix $M^{(k)}$ has dimension $2^k\times 2^k$.

Simply put, each non-zero element $m^{(k)}_{i,j}$ of $M^{(k)}$ relates to the maximal distance each tile of the covering has to the origin of the complex plane. To be more precise, given the collection of tiles $P_1^{(k)},\dots,P_{2^k}^{(k)}$, we have
\begin{equation}\label{eq:non_zero_matrix_elements}
m^{(k)}_{i,j} =
\begin{cases*}
  \left(\max\left\{|p'_c(z)|\colon z\in P_i^{(k)}\cap p_c^{-1}(P_j^{(k)})\right\}\right)^{-1} & if $P_i^{(k)}\cap p_c^{-1}(P_j^{(k)})\neq\emptyset $,  \\
   \phantom{-----} 0 & otherwise.
\end{cases*}
\end{equation}
Following Remark~\ref{metric_remark}, the derivatives $p'_c(z)$ in~(\ref{eq:non_zero_matrix_elements}) are measured in the Euclidean metric. Note that $|p'_c(z)| = 2|z|$ is independent of the parameter $c$. On a technical note, by the nesting properties of the tiles, it follows that we have $P_i^{(k)}\cap p_c^{-1}(P_j^{(k)}) = P_l^{(k+1)}$ for a certain index $l = l(i,j)$. Therefore, when computing $M^{(k)}$ we actually use tiles of the next level $k+1$.

Due to the underlying dynamics of real quadratic polynomials, a McMullen matrix has a four-fold symmetry as follows (we suppress $k$ for now):
\begin{equation}\label{eq:matrix_symmetry}
\sbox0{$\begin{matrix}1&2&3\\0&1&3\\0&0&3\end{matrix}$}
M = \left[
\begin{array}{c|c}
  \vphantom{\usebox{0}}\makebox[\wd0]{\large $G$}&\makebox[\wd0]{\large \rotatebox[origin=c]{180}{$G$}}\\
\hline
  \vphantom{\usebox{0}}\makebox[\wd0]{\large $G$}&\makebox[\wd0]{\rotatebox[origin=c]{180}{\large $G$}}
\end{array}
\right]
\end{equation}
Here the four submatrices all carry exactly the same information. Each of the two right-most matrices (labeled \rotatebox[origin=c]{180}{$G$}) is given by rotating the matrix $G$ 180 degrees. The matrix $G$ has size $2^{k-1}\times 2^{k-1}$; it is sparse and double banded. Out of its $2^{2(k-1)}$ elements, only $2^{k-1}$ are non-zero, which corresponds to one row's worth. In fact, using our encoding of tiles, only the elements of $G$ indexed by $(i, 2i-1)$ and $(i, 2i)$ for $i = 1,\dots,2^{k-2}$ carry any information; all other matrix elements are zero. For efficiency, the non-zero entries of $G$ can be stored in a single $2^{k-1}$-dimensional vector; the elements are computed using  (\ref{eq:non_zero_matrix_elements}). %They depend on the depth $k$ and the parameter $c$ of the quadratic map.

Given a McMullen matrix $M$, together with a positive number $t>0$, we define a new matrix $M(t)$ whose elements are given by
\begin{equation}
m_{i,j}(t) =
\begin{cases*}
  \left(m_{i,j}\right)^{t}   & if $m_{i,j}\neq 0$  \\
                     \phantom{-}0 & if $m_{i,j} = 0$.
\end{cases*}
\end{equation}
In other words, $M(t)$ is the matrix formed by raising each non-zero element of $M$ to the power of $t$. Thus $M(t)$ has exactly the same sparse, banded structure as $M$, see Figure~\ref{fig:matrix_structure_nz}.

\begin{figure}[h]
\begin{center}
\includegraphics[scale=0.41]{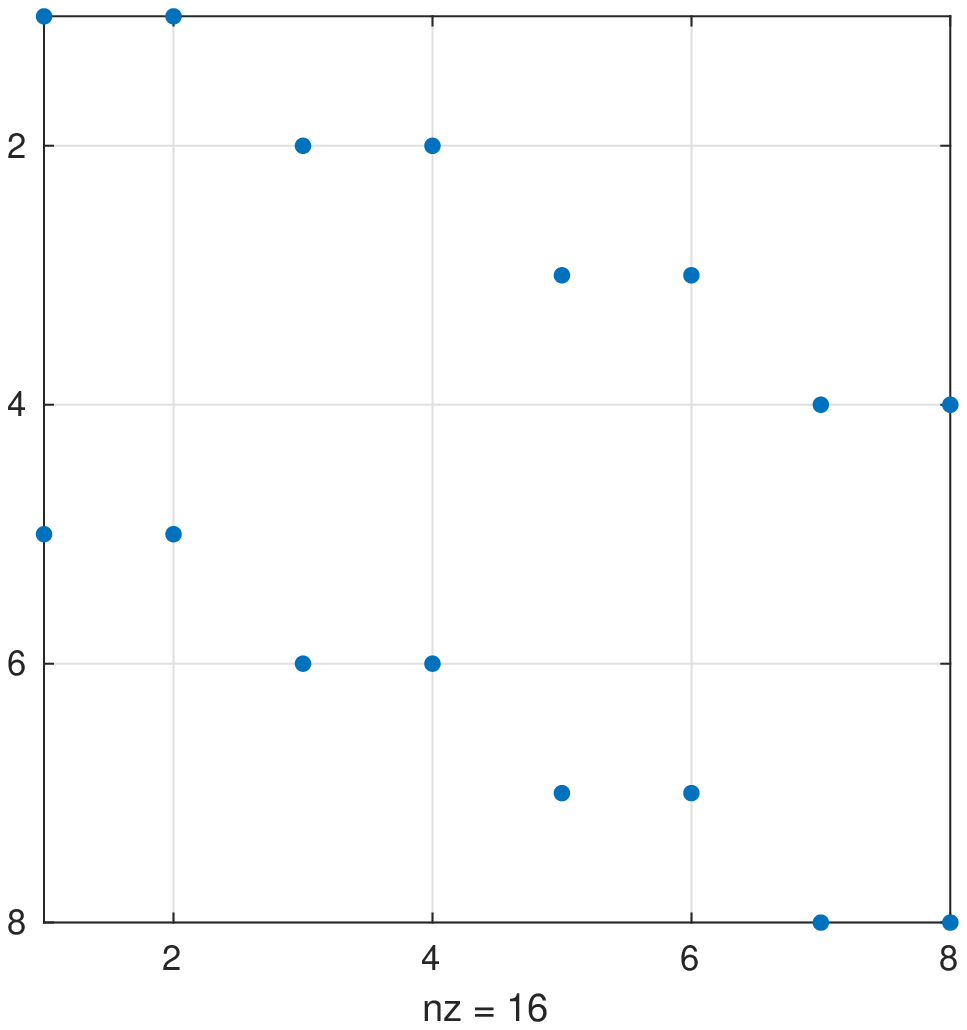}
\hspace*{-12mm}
\includegraphics[scale=0.41]{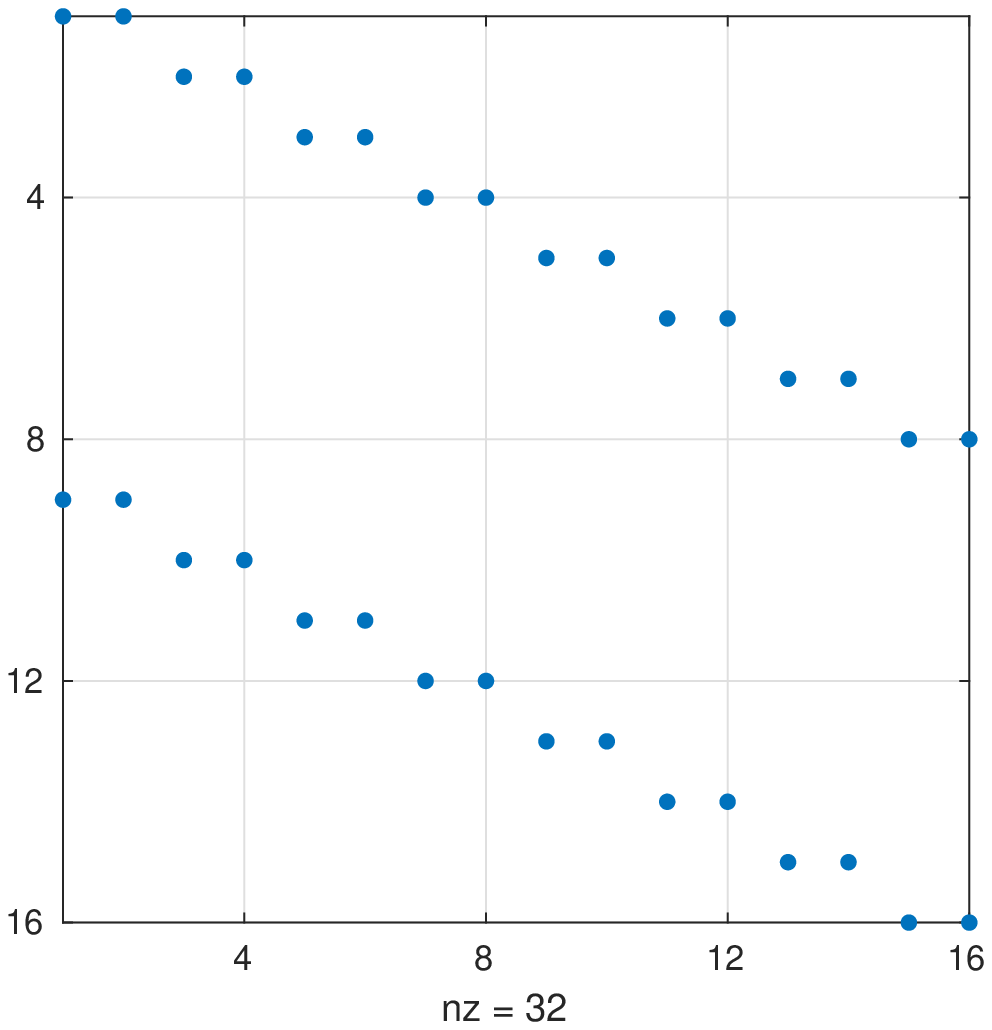}
\hspace*{-12mm}
\includegraphics[scale=0.41]{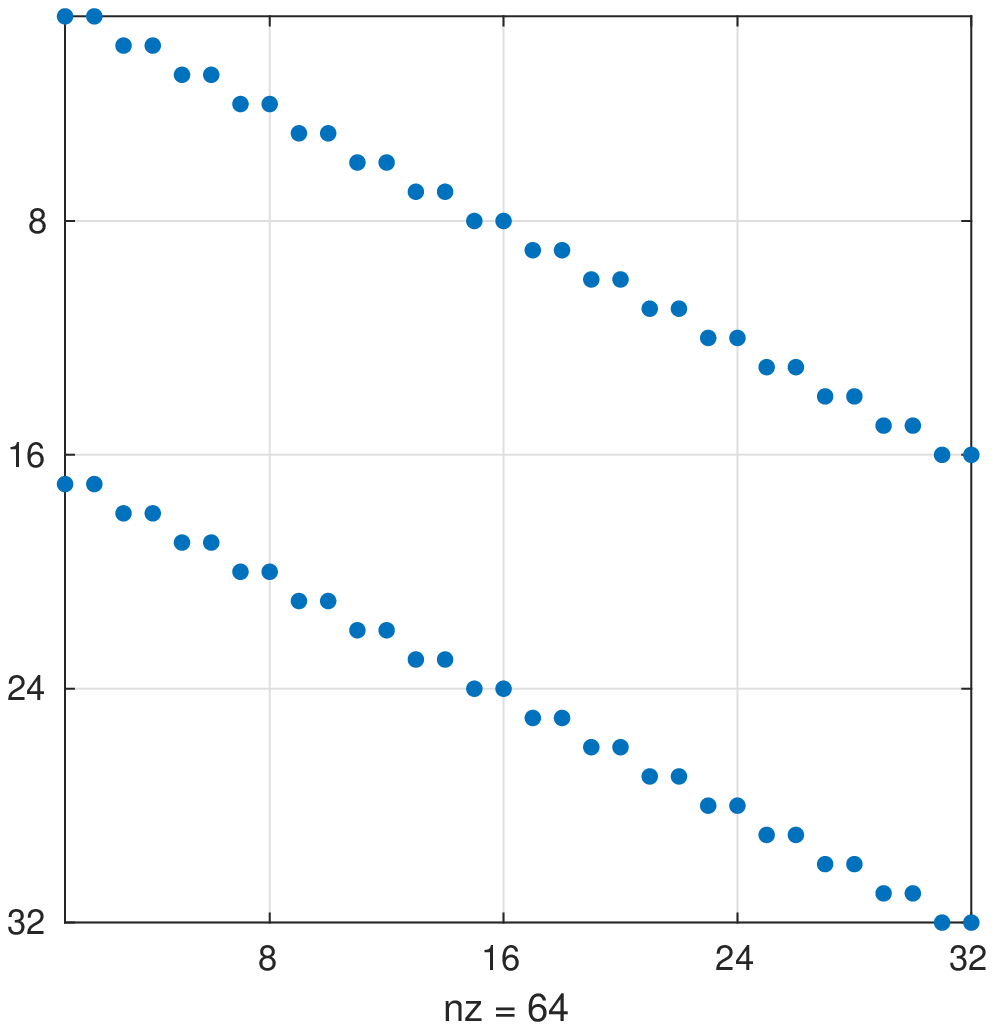}
\end{center}
\caption{The sparse, banded structure of the McMullen matrices $M^{(k)}$ with $2^{k+1}$ non-zero elements. (a) $k = 3$, (b) $k = 4$, and (c) $k = 5$.}
\label{fig:matrix_structure_nz}
\end{figure}

\subsection{Bounding the spectral radius}
\label{subsec:bounding_the_spectral_radius}

For a fixed parameter $c\in [-2,2]$ and a fixed encoding depth $k$ we now focus on the positive number
\begin{equation}\label{eq:definition_of_lo_delta}
\lo{\delta_k} = \inf\{t>0\colon \varrho(M^{(k)}(t)) = 1\}
\end{equation}
which, by Theorem~\ref{Bowen_formula_theorem} and Theorem~\ref{main_inequality_theorem}, we know is a lower bound for the Hausdorff (in fact, hyperbolic) dimension of the Julia set $J_c$. We will continue suppressing the encoding depth $k$, writing $M(\delta)$ in place of $M^{(k)}(\delta)$ and $\lo{\delta}$ in place of $\lo{\delta_k}$. 
Our aim is now to find good lower bounds for $\lo{\delta}$. We will use a very simple two-step strategy: First, we make an educated guess $\delta_\star<2$ for a lower bound for $\lo{\delta}$. Next, we prove that
\begin{equation}\label{want_to_prove_equation}
1< \varrho(M(\delta_\star)).
\end{equation}
Observe that according to Lemma~\ref{pressure_at_2_lemma} combined with~(\ref{pressure_ineq}) and Lemma~\ref{Mktheta_lemma}, it follows that $\varrho(M(2))\le 1$. Together with concavity of the function $t\mapsto \varrho(M^{(k)}(t))$ (see Lemma~\ref{concavity_lemma}) and inequality~(\ref{want_to_prove_equation}), this implies that indeed, $\delta_\star\le\lo{\delta}$, hence, $\delta_\star\le\hypd(J_c)\le \hd(J_c)$.

For this strategy to work, there are three challenges to overcome:
\begin{itemize}
\item[1.] We need some efficient heuristics for making a good (educated) guess for $\delta_\star$.
\item[2.] Given $t>0$, we need to rigorously enclose the matrix elements of $M(t)$.
\item[3.] Given $t>0$, we need to rigorously bound (from below) the spectral radius of $M(t)$.
\end{itemize}

For now, we will assume that we can overcome challenge 2. This will be explained in detail in Section~\ref{subsec:rigorous_tile_computations}.

In order to compute and bound the spectral radius, we will discuss some important classes of matrices.

\begin{definition}
A matrix $A$ is \emph{positive} if all its elements are positive; we then write $A > 0$. Analogously, a matrix $A$ is \emph{non-negative} if all its elements are non-negative: we then write $A \ge 0$.
\end{definition}

Using this notation we can write e.g. $A \ge B$, meaning that $A-B\ge 0$.

\begin{definition}
A square, non-negative matrix $A$ is \emph{primitive} if there exists a positive integer $n$ such that $A^n$ is positive.
\end{definition}
A related, but somewhat weaker, property is that of being \textit{irreducible}. All primitive matrices are irreducible. The latter class is important since the Perron-Frobenius theory for positive matrices generalizes to the class of irreducible matrices.

We point out that, by construction, all McMullen matrices $M$ and $M(t)$ are primitive (and thus irreducible).

By a classic result by Collatz \cite{Co42} we have the following:
\begin{theorem}\label{thm:collatz}
Let $A$ be a non-negative, irreducible, $n\times n$ matrix, and let $v^{(0)}$ be an arbitrary positive $n$-dimensional vector. Defining $v^{(j+1)} = [v^{(j+1)}_1\, \dots \,v^{(j+1)}_n]^T = Av^{(j)} = \dots = A^{j+1}v^{(0)}$, and setting
\begin{equation}
\lo{\lambda}^{(j)} = \min_{1\le i\le n}\left(\frac{v_i^{(j+1)}}{v_i^{(j)}}\right)
\qquad\textrm{ and }\qquad
\hi{\lambda}^{(j)} = \max_{1\le i\le n}\left(\frac{v_i^{(j+1)}}{v_i^{(j)}}\right)
\end{equation}
we have the following sequence of enclosures
\begin{equation}\label{eq:spectral_enclosures1}
\lo{\lambda}^{(0)}\le\lo{\lambda}^{(1)}\le\dots\le \varrho(A)\le\dots\le\hi{\lambda}^{(1)}\le\hi{\lambda}^{(0)}.
\end{equation}
\end{theorem}
This theorem is very useful: it \textit{encloses} the spectral radius of $A$, and therefore we can use any point $r \in [\lo{\lambda}^{(j)}, \hi{\lambda}^{(j)}]$ as an approximation of the spectral radius. Furthermore, we have the bound $|r - \varrho(A)| \le \hi{\lambda}^{(j)} - \lo{\lambda}^{(j)}$ on the maximal approximation error we are making by using $r$ in place of $\varrho(A)$. When the matrix $A$ is primitive, it is known that this error tends to zero as $j$ increases; $\lim_{j\to\infty} (\hi{\lambda}^{(j)} - \lo{\lambda}^{(j)}) = 0$.

We are now prepared to overcome challenge 1: finding a good candidate for $\delta_\star$ -- the lower bound for $\lo{\delta}$ appearing in (\ref{eq:definition_of_lo_delta}). We request that $\delta_\star$ satisfies two bounds:
\begin{equation}\label{eq:bounds_on_delta_star}
0< \varrho(M(\delta_\star)) - 1 \le \varepsilon,
\end{equation}
where $\varepsilon$ is a small positive number. The smaller we take $\varepsilon$, the better our bound becomes. The leftmost inequality of (\ref{eq:bounds_on_delta_star}), which is the same as inequality~(\ref{want_to_prove_equation}), ensures that $\delta_\star$ is a true lower bound (and not just an approximation of $\lo\delta$). Using Theorem~\ref{thm:collatz}, we can find a good $\delta_\star$ via a simple bisection scheme. 
A straight-forward computation indicates that $\varrho(M(2)) < 1 < \varrho(M(\tfrac{1}{10}))$. Based on this, we start the bisection search on the interval $\Delta_0 = [\tfrac{1}{10},2]$. After $j$ steps we have an interval $\Delta_j$ of width less than $2^{1-j}$, whose left endpoint is a valid lower bound $\delta_\star$ of $\lo{\delta}$. In our computations, we end the bisection scheme when we have reached $\varepsilon \le 10^{-10}$ in (\ref{eq:bounds_on_delta_star}).

\subsection{Set-valued computations}
\label{subsec:set-valued_computations}

For efficiency reasons, all computations carried out in the bisection scheme are performed using normal floating point arithmetic. As such they are not entirely reliable: rounding errors can accumulate, and we must therefore carefully verify that $\delta_\star$ indeed is a true lower bound of high accuracy. For the lower bounds reported in Theorem~\ref{thm:lower_bound_on_hd} and Theorem~\ref{thm:lower_bound_on_hd_c}, we use an encoding depth of 28, which means that $M$ has $2^{29} = 536~870~912$ non-zero elements. For Theorem~\ref{thm:uniform_bound_on_hd} we use a mix of depths 17 and 18.

In order to certify the computations described in Section~\ref{subsec:bounding_the_spectral_radius}, we have to resort to set-valued numerics. We will work with interval-valued matrices, and all operations will use interval arithmetic with outward directed rounding. An elementary overview of the theory of interval analysis is given in \cite{Moore1966,Tucker2011book}.

For now, we assume that we are given interval-valued McMullen matrices $\ivMM^{(k)}$, whose non-zero elements are intervals: $\ivM^{(k)}_{i,j} = [\lo{m}^{(k)}_{i,j}, \hi{m}^{(k)}_{i,j}]$. The width of each such interval reflects the uncertainty in the underlying tile construction; this is described at depth in  Section~\ref{subsec:rigorous_tile_computations}. Again, we will continue to suppress the encoding depth $k$ for clarity.

Given an interval McMullen matrix $\ivMM$ we want to find a good lower bound for the hyperbolic dimension of the Julia set $J_c$:
\begin{equation}\label{eq:definition_of_lo_delta_2}
\lo{\delta} = \inf\{t>0\colon \varrho(M(t))=1, \textrm{ for some } M(t)\in \ivMM(t)\}.
\end{equation}

To this end, we will use the following result (see \cite{Nu86}):
\begin{lemma}\label{lem:ordered_spectral_radii}
If $A$ and $B$ are irreducible matrices such that $0\le A \le B$, then $\varrho(A) \le \varrho(B)$.
\end{lemma}

Recall that, by construction, each McMullen matrix is primitive (and therefore irreducible). This property carries over to the interval-valued versions $\ivMM(t)$ too. Let $\lo{M}(t) = \min(\ivMM(t))$ denote the matrix whose elements are formed by taking the lower endpoint of each non-zero (interval) element of $\ivMM(t)$. Then, by Lemma~\ref{lem:ordered_spectral_radii}, it follows that
\begin{equation}
\varrho(\lo{M}(t)) = \min\left\{\varrho(A) \colon A \in \ivMM(t)\right\}.
\end{equation}
We now validate our (educated) guess $\delta_\star$ by computing $\varrho(\lo{M}(\delta_\star))$ using Theorem~\ref{thm:collatz}, but with all operations carried out in interval arithmetic. This will produce a rigorous enclosure of $\varrho(\lo{M}(\delta_\star))$ which can be used to validate the lower bound $\delta_\star$. More specifically, we check that $\varrho(\lo{M}(\delta_\star))>1$, which implies~(\ref{want_to_prove_equation}), which in turn, implies that $\delta_\star$ is a true lower bound on $\lo{\delta}$.

\subsection{Rigorous tile computations}
\label{subsec:rigorous_tile_computations}

Recall the notation from Section~\ref{Partitions_section}: let $D_r(z)$ denote the open disc of radius $r\ge 0$ centered at $z\in\C$. Note that $D_2(0)\setminus [-2,2]$ consists of two half-discs. Given a quadratic map $p_c(z) = z^2 + c$ with $c\in[-2,2]$ and an integer $k=2,3,\dots$, we will let $\calP^{(k)}$ denote the set of connected components of $p_c^{1-k}(D_2(0)\setminus [-2,2])$. The elements of $\calP^{(k)}$ are called \textit{tiles}; there are exactly $2^{k}$ of them at level $k$. When $c$ is real, the tiles have a four-fold symmetry: each quadrant is a reflection of the others. When $c$ is complex, the symmetry is only two-fold (see Figure~\ref{fig:level_3_real_c}), and the restriction of the map $p_c$ to each of the tiles no longer provides a partitioned holomorphic dynamical system as the Markov property fails.

\begin{figure}[h]
\begin{center}
\includegraphics[scale=0.6]{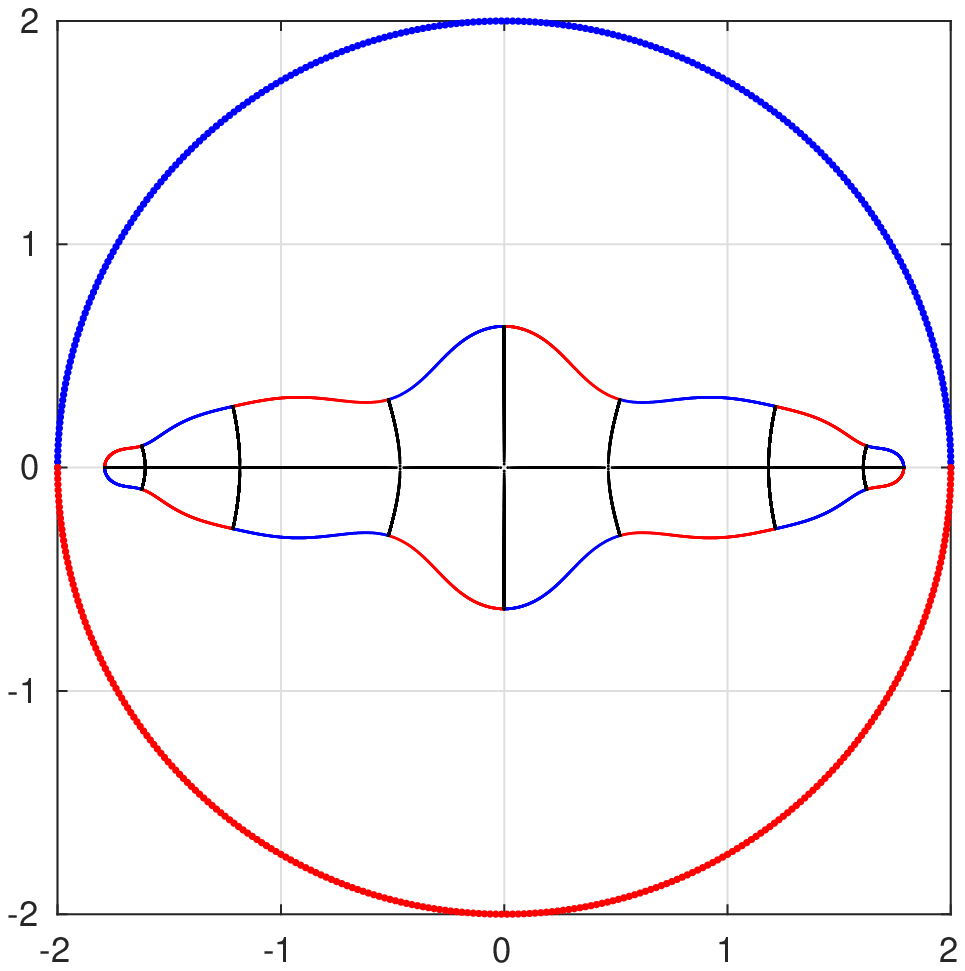}
\hspace*{-18mm}
\includegraphics[scale=0.6]{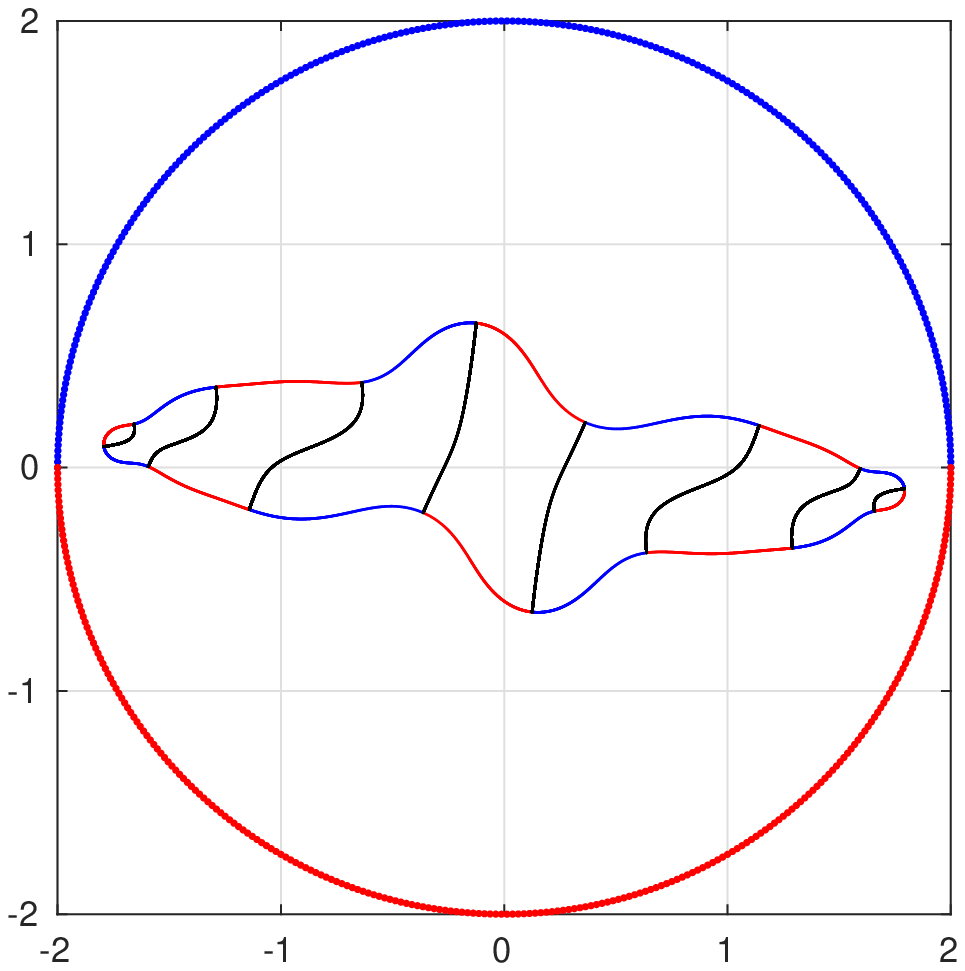}
\end{center}
\caption{All $16$ tiles of $\calP^{(4)}$ for two different quadratic maps. (a) For real parameters (here $c = c_\star$) there is a four-fold symmetry. (b) For complex parameters (here $c = c_\star + \tfrac{1}{4}i$) a two-fold symmetry remains.}
\label{fig:level_3_real_c}
\end{figure}

Suppressing the level $k$ for the moment, recall that, for each tile $P_i$, we want to compute the maximal distance $s_i = \max\{|z|\colon z\in P_i\}$. Since $p_c'(z) = 2z$, we have the relation $m_{i,j} = 1/(2s_i)$, whenever $P_j\subset p_c(P_i)$, see (\ref{eq:non_zero_matrix_elements}) of Section~\ref{subsec:constructing_the_mcmullen_matrices}. In order to get rigorous enclosures of the elements of the McMullen matrices, we must be able to enclose the distances $s_i$ for all tiles at a fixed level. We do this by covering the boundary of each tile with small discs, starting with the two initial half-discs $D_2(0)\setminus [-2,2]$. From these initial half-discs, we generate the entire collection of tiles at level $k$ is by repeatedly applying the inverse map $p^{-1}_c(w) = \pm\sqrt{w - c}$. This of course, requires an inverse map that is disc-valued.

\begin{figure}[h]
\begin{center}
\includegraphics[scale=0.6]{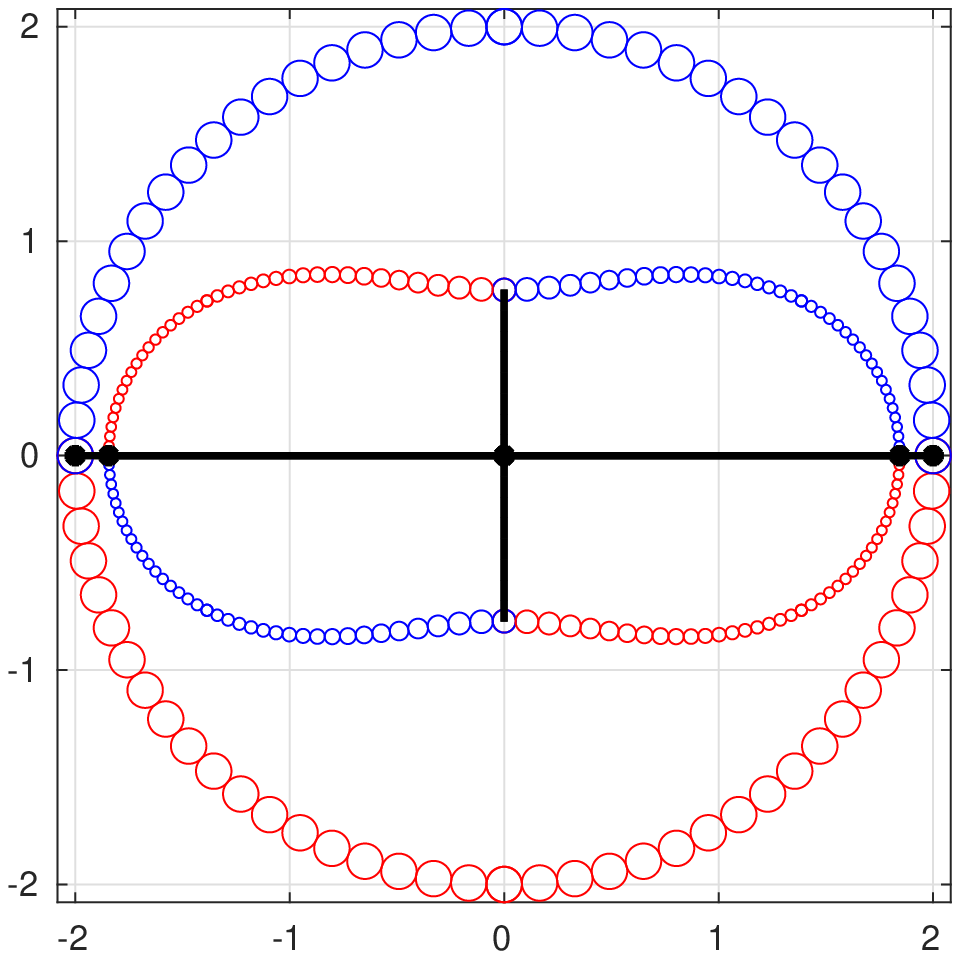}
\hspace*{-18mm}
\includegraphics[scale=0.6]{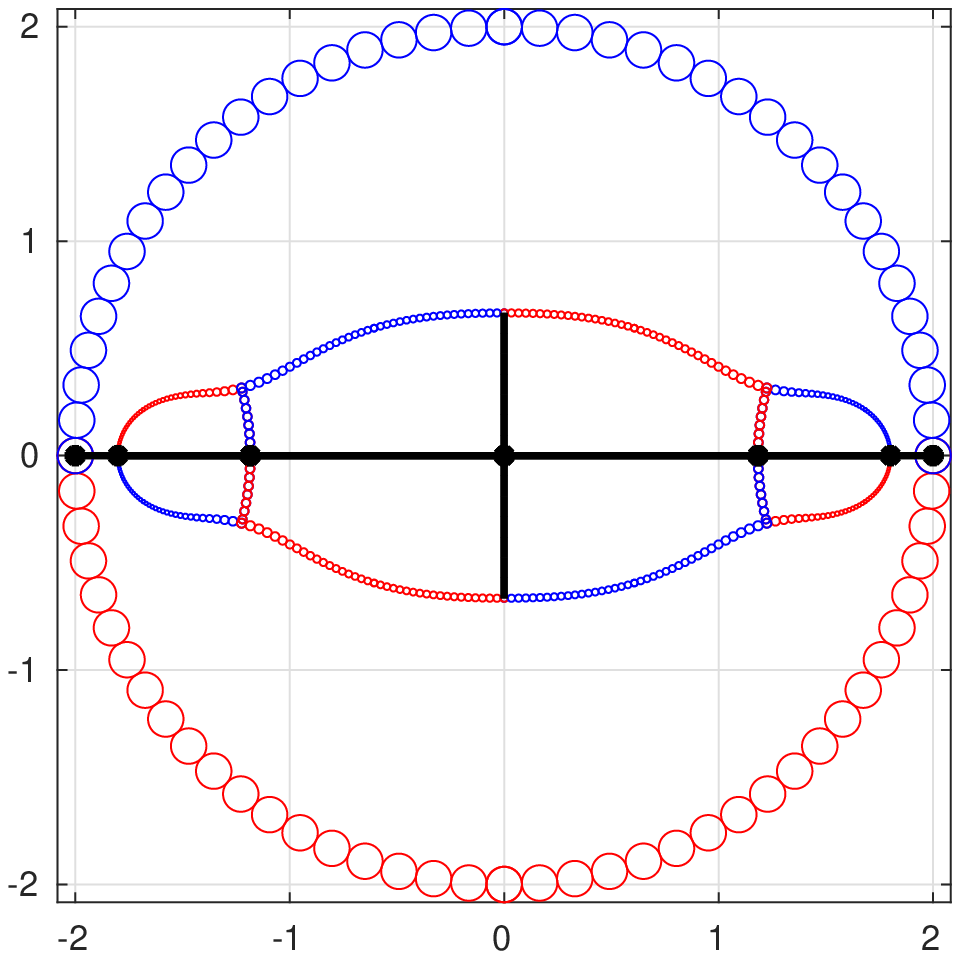}
\end{center}
\caption{The two first levels of tiles for the parameter set $\ivC=\langle c_\star, 10^{-10}\rangle$ using 39 discs to cover each initial half-disc. The colour coding indicates where a tile emanates from: blue -- upper half plane; red -- lower half plane.  (a) $\calP^{(2)}$ has four tiles. (b) $\calP^{(3)}$ has eight tiles. }
\label{fig:level_2_wide}
\end{figure}

We will describe how this is achieved in detail in Sections~\ref{subsec:circular_complex_arithmetic} and ~\ref{subsec:extended_circular_complex_arithmetic}. For now, let us assume that we can cover the boundary of each tile $P_i$ by a finite collection of small discs $d_{i,j}$, see Figure~\ref{fig:level_2_wide}. Computing upper and lower bounds ($\hi{s}_i$ and $\lo{s}_i$ , respectively) for each maximal distance $s_i$ is then a simple matter of traversing the collection of discs, and finding the maxmin/maxmax distance to the origin:
\begin{eqnarray}
\lo{s}_i &=& \maxmin_j\left\{\mid(d_{i,j}) - \rad(d_{i,j}) \right\},\\
\hi{s}_i &=& \maxmax_j\left\{\mid(d_{i,j}) + \rad(d_{i,j}) \right\}.
\end{eqnarray}

We end this section with a practical note. Each tile $P_i^{(k)}\in\calP^{(k)}$ is encoded via the binary representation of its index $(i)_{10} = (b_{k-1} \dots b_0)_2$, where $i=0,1,\dots, 2^k-1$. The encoding is used to select the correct branch of the inverse image: if $b_j = 0$, then the $j$:th inverse image belongs to the lower half plane. If $b_j = 1$, we select the inverse image in the upper half plane. This produces a labeling of tiles starting at the right-most position of the fourth quadrant and moving clock-wise through the 3:rd, 2:nd, and 1:st quadrant. By the four-fold symmetry, it suffices to compute the first quarter of tiles. They all belong to the fourth quadrant, and are indexed $i=0,\dots, 2^{k-2}-1$.

\subsection{Circular complex arithmetic}
\label{subsec:circular_complex_arithmetic}

There are several models for set-valued complex arithmetic: the fundamental choice lies in the the way a set of complex numbers is represented. A survey of different approaches can be found in \cite{ComplexBook}.
We will opt for circular sets; this will minimize the overestimation when computing the inverse images of the quadratic map.

A set in the complex plane is thus represented by a disc with midpoint $m\in\C$ and radius $r\ge 0$:
\begin{equation}
\ivZ = \langle m,r \rangle = \{z\in\C\colon |z-m|\le r\}.
\end{equation}
Computing the inverse map $p_c^{-1}(w) = \sqrt{z - c}$ involves two operations: subtraction and taking the square root. Subtracting a disk from another is straight-forward. We have
\begin{equation}\label{eq:disc_subtraction}
\ivZ_1 - \ivZ_2 = \langle m_1,r_1 \rangle- \langle m_2,r_2 \rangle = \langle m_1 - m_2, r_1 + r_2 \rangle,
\end{equation}
which incurs no overestimation: the resulting disc is a sharp enclosure of the difference. The square root, however, is more complicated as discs are not preserved under this operation. Given an input disc $\ivZ$ we must find a new disc $\ivW$ that contains the set $\{\sqrt{z} \colon z\in\ivZ\}$. We will only focus on one of the two branches of the square root (the second branch is simply the negative of the first).

First, we write the midpoint of $\ivZ = \langle m,r \rangle$ in polar coordinates: $m = \rho e^{i\theta}$. Assuming (for now) that $0\notin \ivZ$, we have $0\le r < \rho$. From the two positive numbers $\alpha_1 = \sqrt{\rho + r}$ and $\alpha_2 = \sqrt{\rho - r}$, we form $\tilde\rho = \frac{\alpha_1 + \alpha_2}{2}$ and $\tilde r = \frac{\alpha_1 - \alpha_2}{2}$. We then have the enclosure of the square root
\begin{equation}\label{eq:disc_sqrt}
\sqrt{\ivZ} = \left\langle \tilde\rho e^{i\theta /2}, \tilde r \right\rangle,
\end{equation}
which is the sharpest possible. When the ratio $r/\rho$ is small, the overestimation is very small. As the ratio approaches one, however, there is a significant difference between the true image of the square root and its disc enclosure given by (\ref{eq:disc_sqrt}). This is illustrated in Figure~\ref{fig:sqrt_disc}.

\begin{figure}[h]
\begin{center}
\includegraphics[scale=0.35]{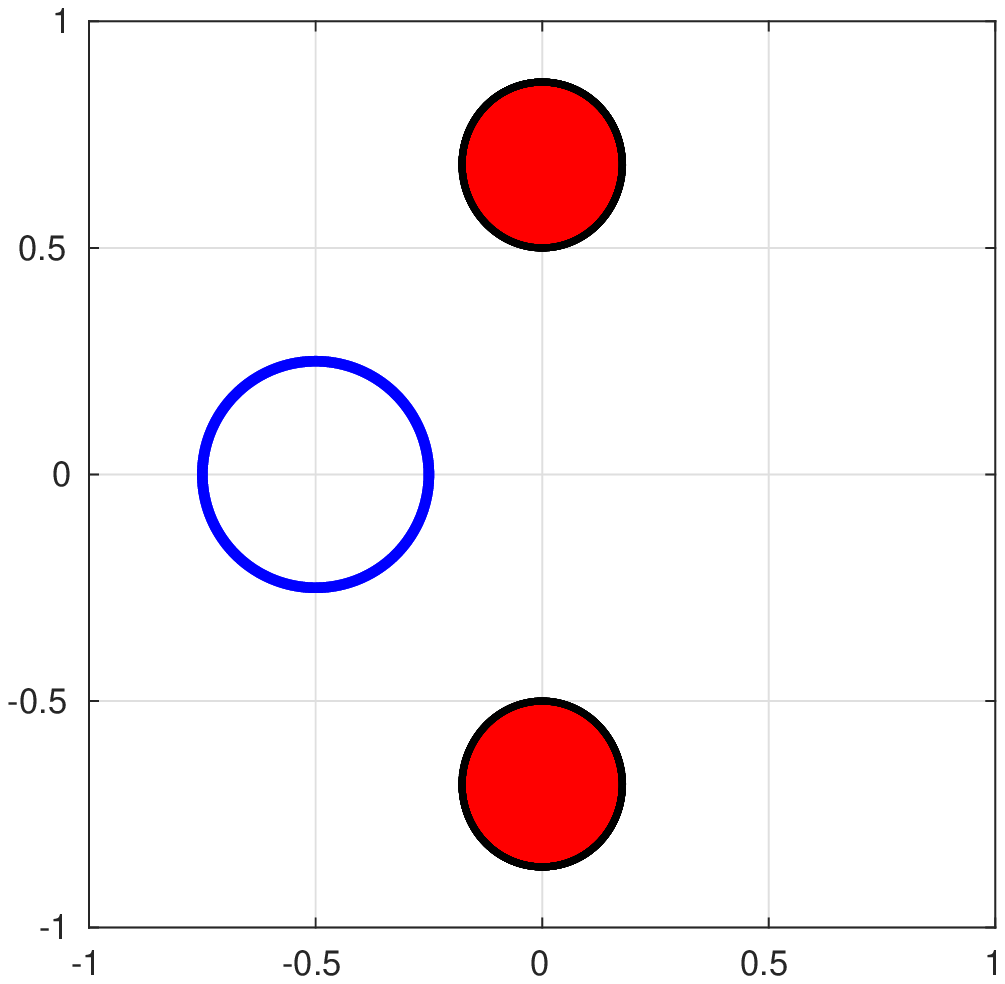}
\hspace*{-5mm}
\includegraphics[scale=0.35]{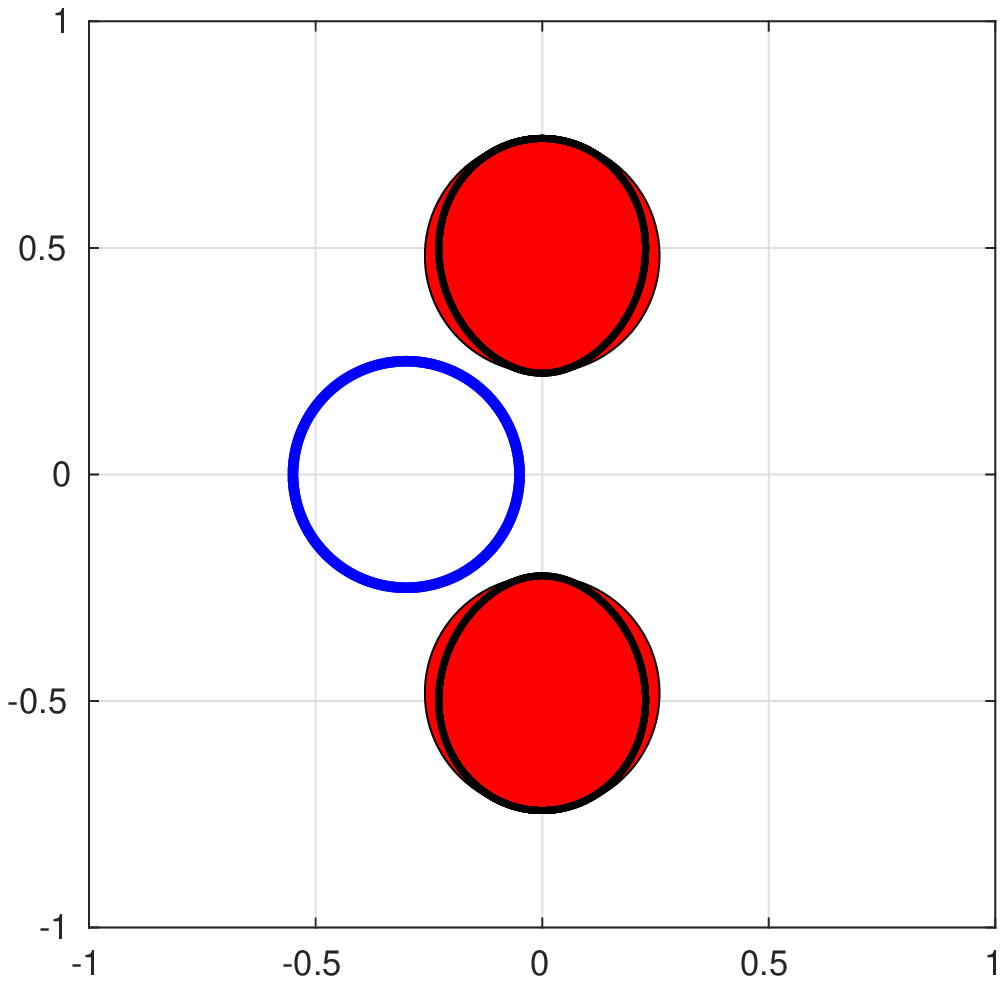}
\hspace*{-5mm}
\includegraphics[scale=0.35]{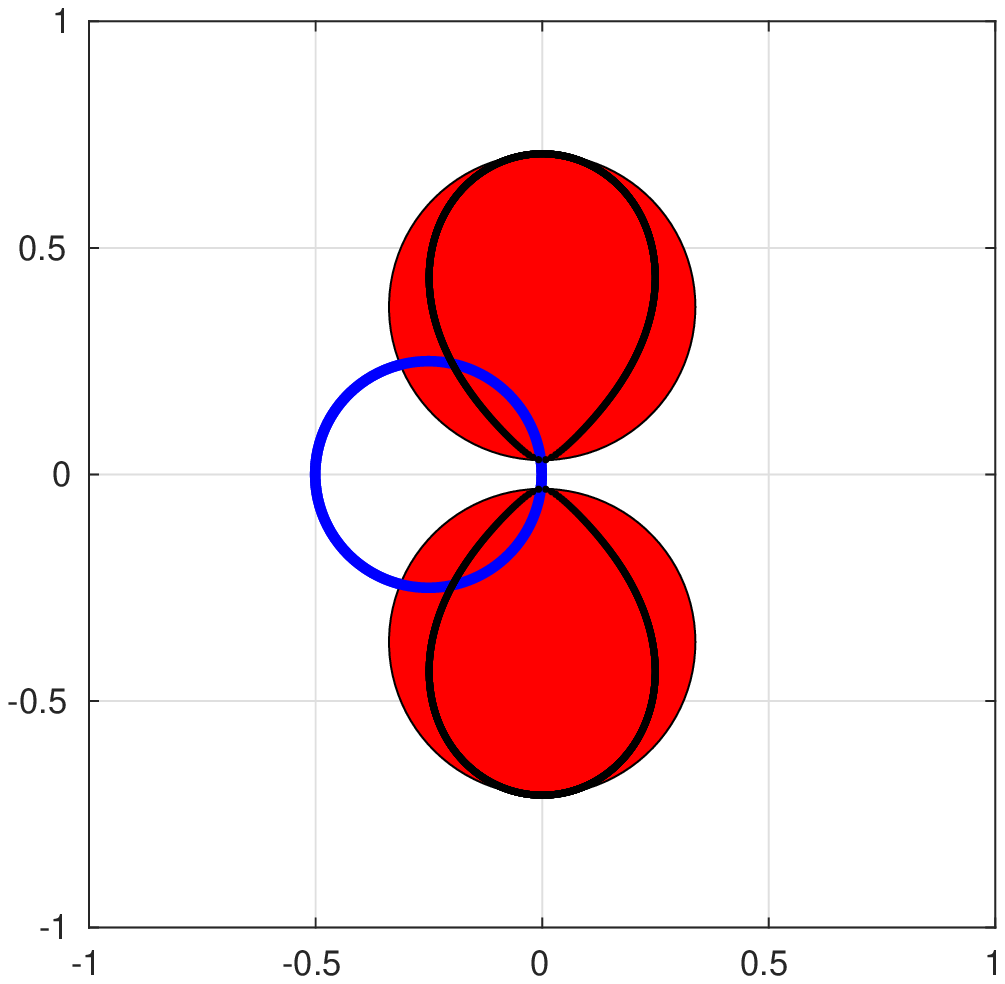}
\end{center}
\caption{Both branches of the complex square root; a comparison between the enclosures (red discs) given by (\ref{eq:disc_sqrt}) and the true image (the interior of the black curves). The domain is the blue disc $\langle m, r \rangle$ with $r = 1/4$ and a varying midpoint: (a) $m = -0.5$, (b) $m = -0.3$, and (c) $m = -0.251$.}
\label{fig:sqrt_disc}
\end{figure}

When the disc $\ivZ = \langle m,r \rangle$ contains the origin, the situation changes drastically. Throughout our computations, however, this will only occur in very special situations. We will handle these special cases by extending our circular complex arithmetic.

\subsection{Extended circular complex arithmetic}
\label{subsec:extended_circular_complex_arithmetic}

Considering the computations needed to establish our results, we want to extend the circular arithmetic introduced in Section~\ref{subsec:circular_complex_arithmetic}. In order to compute the tiles efficiently (see Section~\ref{subsec:rigorous_tile_computations}) we would like to be able to represent (and compute with) interval segments of the real and imaginary axes. Indeed, part of the boundary of many tiles are subsets of the real and imaginary axes, see Figure~\ref{fig:level_3_real_c}. To this end, we will extend the notion of a disc to include such subsets of the complex plane. We introduce the following notation:
\begin{eqnarray}\label{eq:extended_disc}
\langle m,r\rangle_x &=& \{z\in\C\colon |z-m|\le r \textrm{ and } \Im(z) = 0 \},\nonumber\\
\langle m,r\rangle_y &=& \{z\in\C\colon |z-m|\le r \textrm{ and } \Re(z) = 0 \}.
\end{eqnarray}
These \textit{collapsed} discs contain only real or imaginary numbers, respectively. We can now define precise sets of real parameters for which we will prove our theorems. As an example, for Theorem~\ref{thm:lower_bound_on_hd} and Theorem~\ref{thm:lower_bound_on_hd_c} , we will work on very small subsets of the real axis: $\ivC = \langle c_\star, 10^{-10}\rangle_x$. For Theorem~\ref{thm:uniform_bound_on_hd} we use intervals having much larger radii. In each inverse iteration of the quadratic map, we will be subtracting $\ivC$ from a disc. This operation is straight-forward; given two complex numbers $z_1 = x_1 + iy_1$ and $z_2 = x_2 + iy_2$, we define
\begin{eqnarray}\label{eq:extended_disc_subtraction}
\langle z_1,r_1 \rangle_x - \langle z_2,r_2 \rangle_x &=& \langle x_1-x_2, r_1 + r_2 \rangle_x, \\
\langle z_1,r_1 \rangle_y - \langle z_2,r_2 \rangle_x &=& \langle -x_2 + iy_1, r_1 + r_2 \rangle, \\
\langle z_1,r_1 \rangle\phantom{x} - \langle z_2,r_2 \rangle_x &=& \langle x_1-x_2+iy_1, r_1 + r_2 \rangle.
\end{eqnarray}
There are other combinations of operands, but these three are the only ones we will use.

We will also need to compute the square root of discs of type $\langle m,r\rangle_x$.  To simplify notation, let $\check{x}$ be a real number and consider the disc $\ivX = \langle \check{x},r\rangle_x$. There are three different situations that we must account for.

If $0\le \check{x} - r$, then the disc $\ivX$ consists of non-negative numbers only. We can use the same formula as (\ref{eq:disc_sqrt}) with $\theta = 0$, and making sure we stay in the correct class of discs:
\begin{equation}\label{eq:disc_sqrt3}
\sqrt{\ivX} = \sqrt{\langle \check{x},r\rangle_x} = \left\langle \tilde\rho, \tilde r \right\rangle_x.
\end{equation}
If $\check{x} + r\le 0$, then the disc $\ivX$ consists of non-positive numbers only. Once again, we can use the same formula as (\ref{eq:disc_sqrt}) with $\theta = \pi/2$. This corresponds to a rotation from $\theta = \pi$ to $\pi/2$, which brings us to the second class of extended discs:
\begin{equation}\label{eq:disc_sqrt4}
\sqrt{\ivX} = \sqrt{\langle \check{x},r\rangle_x} = \left\langle i\tilde\rho, \tilde r \right\rangle_y.
\end{equation}
Finally, if $\ivX$ contains zero in its interior, we split the disc into its non-negative and non-positive counterparts, and treat them according to (\ref{eq:disc_sqrt3}) and (\ref{eq:disc_sqrt4}), respectively.

This concludes our derivation of disc-valued arithmetic. In the actual numerical computations (see Section~\ref{subsec:computer_outputs}), all floating point operations are performed with outward directed rounding. This ensures that all results are true enclosures and can be used as rigorous bounds.

\subsection{Computer outputs}
\label{subsec:computer_outputs}

In this section, we will present the numerical results behind Theorem~\ref{thm:lower_bound_on_hd}, Theorem~\ref{thm:lower_bound_on_hd_c} and Theorem~\ref{thm:uniform_bound_on_hd}. For the basic interval arithmetic, we use the CAPD library (\cite{KapelaMrozekWilczakZgliczynski2020}), which is well established in the community of computer-assisted proofs in dynamical systems. The code we have developed, however, can easily be adapted to any other interval library. The main program is based on the ideas presented above, and consists of three main stages:
\begin{itemize}
\item[1.] Compute the tiles and the McMullen matrix elements.
\item[2.] Compute an approximation of a lower bound for the hyperbolic dimension.
\item[3.] Validate the lower bound.
\end{itemize}

In order to speed up the computations, we perform a large part on them in parallel. Indeed, each tile can be computed independently, and all matrix-vector products required in the implementation of Theorem~\ref{thm:collatz} benefit from parallelization.

Starting with Theorem~\ref{thm:lower_bound_on_hd}, we launch our program aiming at depth $k = 28$, with an initial covering of each half-disk made up of 199 full discs and one collapsed disc. The output is presented below.

\begin{verbatim}
------------------------------------------------------------
-------------------------- hypdim --------------------------
Specifications:
Running on host 'ha10' with 12 CPUs and 12 threads.
Started computations: Sun Sep 19 01:11:56 2021
------------------------------------------------------------
Tile computations:
Covering the half circles with 199 discs each.
Covering the x-axis with 1 collapsed discs.
Splitting the y-axis with tol: 0.01
c = {-1.401155189,0;1e-10}, dt = 0
Computing 67108864 tiles to depth 28
maxDiam(MM) = 0.093130430392924346
Wall time: 16 hours, 37 minutes, and 19 seconds.
CPU time : 198 hours, 37 minutes, and 15 seconds.
------------------------------------------------------------
Non-rigorous lower bound:
tol = 1.0e-10
hd_lo = 1.49781090621836621e+00
Wall time: 3 hours, 20 minutes, and 50 seconds.
CPU time : 5 hours, 42 minutes, and 37 seconds.
------------------------------------------------------------
Confirm lower bound:
Confirmed that HD >= 1.49781090621836621
Wall time: 0 hours, 15 minutes, and 52 seconds.
CPU time : 0 hours, 28 minutes, and 41 seconds.
------------------------------------------------------------
Ended computations: Sun Sep 19 21:25:58 2021
------------------------------------------------------------
\end{verbatim}

Summarizing, the entire run takes about 20 hours, using 12 threads. In forming the interval-valued McMullen matrix $\ivMM$, we compute $2^{26} = 67,108,864$ tiles. This part of the run is the most demanding: 80\% of the run-time is spent on this stage, despite the fact that this is where we gain most on parallelisation. For each tile, we compute rigorous bounds on the distance to the origin using the disc coverings, and then form the matrix entries as explained in Section~\ref{subsec:rigorous_tile_computations}. We note that the widest interval appearing in the McMullen matrix has diameter $0.09313...$. While computing the tiles, we split every produced disc of type $\langle m,r\rangle_y$ until the radius is less than $0.01$. This improves the quality of the tile covering by reducing the size of certain discs. This, in turn, prevents the matrix elements from becoming too wide.

The second stage, based on Section~\ref{subsec:bounding_the_spectral_radius} takes over three hours, and produces a good quality "guess" for a lower bound on the Hausdorff dimension. This is attained by a non-rigorous bisection scheme producing increasingly accurate estimates.

Finally, the third stage (again fully rigorous) validates the guess, and proves that the bound given in Theorem~\ref{thm:lower_bound_on_hd} is a true lower bound.

Continuing to Theorem~\ref{thm:lower_bound_on_hd_c}, we do exactly the same as above, but on a set of different parameters. We will not display the output from all ten runs; one typical run will suffice. Here we use the same depth $k = 28$, but reduce the initial covering of each half-disk: we use 39 full discs and one collapsed disc.

\begin{verbatim}
------------------------------------------------------------
-------------------------- hypdim --------------------------
Specifications:
Running on host 'abel' with 64 CPUs and 64 threads.
Started computations: Tue Mar 29 01:18:36 2022
------------------------------------------------------------
Tile computations:
Covering the half circles with 39 discs each.
Covering the x-axis with 1 collapsed discs.
Splitting the y-axis with tol: 0.01
c = {-1.7812168061,0;1e-10}, dt = 0
Computing 67108864 tiles to depth 28
maxDiam(MM) = 0.45503957706386089
Wall time: 3 hours, 47 minutes, and 0 seconds.
CPU time : 221 hours, 58 minutes, and 41 seconds.
------------------------------------------------------------
Non-rigorous lower bound:
tol = 1.0e-10
hd_lo = 1.32518996605940753e+00
Wall time: 6 hours, 10 minutes, and 29 seconds.
CPU time : 35 hours, 2 minutes, and 24 seconds.
------------------------------------------------------------
Confirm lower bound:
Confirmed that HD >= 1.32518996605940753
Wall time: 0 hours, 20 minutes, and 7 seconds.
CPU time : 2 hours, 30 minutes, and 55 seconds.
------------------------------------------------------------
Ended computations: Tue Mar 29 11:36:14 2022
------------------------------------------------------------
\end{verbatim}
Summarizing, the entire run takes about 10 hours, using 64 threads. Note that due to the higher number of threads, stage one now executes much faster than the second stage.

Turning now to Theorem~\ref{thm:uniform_bound_on_hd}, we recall that the largest parabolic parameter in $(c_\feig, -3/4)$ is $\hat c_2 = -5/4$, and the second largest is $\hat c_3\approx -1.3680989$. Therefore we can cover all parabolic parameters in $(c_\feig, -3/4)$ by the two intervals $I_1 = [-1.2501, -1.2499]$ and $I_2 = [-1.402, -1.350]$.

Note that $I_1 = \langle -5/4, 10^{-4} \rangle_x$, which we can treat as a small parameter interval. We launch our program aiming at depth $k = 18$, with an initial covering of each half-disk made up of 39 full discs and one extended disc. The output is presented below.

\begin{verbatim}
------------------------------------------------------------
-------------------------- hypdim --------------------------
Specifications:
Running on host 'wbt-asus' with 4 CPUs and 4 threads.
Started computations: Thu Dec  9 17:49:22 2021
------------------------------------------------------------
Tile computations:
Covering the half circles with 39 discs each.
Covering the x-axis with 1 collapsed discs.
Splitting the y-axis with tol: 0.01
c = {-1.25,0;1e-05}, dt = 0
Computing 65536 tiles to depth 18
maxDiam(MM) = 0.019450062619993824
Wall time: 49.8641 seconds.
CPU time : 192.2570 seconds.
------------------------------------------------------------
Non-rigorous lower bound:
tol = 1.0e-10
hd_lo = 1.33745732508541693e+00
Wall time: 1.0529 seconds.
CPU time : 4.0847 seconds.
------------------------------------------------------------
Confirm lower bound:
Confirmed that HD >= 1.33745732508541693
Wall time: 0.1742 seconds.
CPU time : 0.6625 seconds.
------------------------------------------------------------
Ended computations: Thu Dec  9 17:50:13 2021
------------------------------------------------------------
\end{verbatim}

This computation proves that $\hd(J_{\hat c_2}) > 4/3$. The remaining parabolic parameters, which all belong to the interval $I_2 = [-1.402, -1.350]$, are treated much the same. The only difference is that we split $I_2$ into 26 smaller subintervals (all of radius $10^{-3}$) and reduce the depth to $k = 17$. This is sufficient in order to obtain good enough lower bounds on the Hausdorff dimensions. We will not report all verbatim output from these computations; the lower bounds ranged from $1.33929890631116133$ to $1.35502280450309742$ (all greater than $4/3$). The total computation time was 4 minutes and 15 seconds when run on 12 cores in parallel.

Finally, we redo the computations above with a bit more effort, and over a wider parameter domain. To be more specific: we split the domain $[-2, 2]$ into 21000 subintervals (of varying radii), each on which we compute a rigorous lower bound for $\hypd(J_c)$. This produces a piecewise constant function acting as a lower bound on the hyperbolic dimension of the Julia sets for the full domain $[-2, 2]$, and is illustrated in Figure~\ref{fig:stairs_plot}(a). The underlying computations took three days on ten 12-core machines. Using only a subset of this information, we can (again) prove Theorem~\ref{thm:uniform_bound_on_hd} by showing that the hyperbolic dimension of the Julia set satisfies $\hypd(J_c) > 4/3$ over the entire interval $[-1.53, -1.23]$, see Figure~\ref{fig:level_14_real_c}(a). Since this interval contains all real parabolic parameters $\hat c\in (c_{\feig}, -3/4)$, the theorem follows.

We end this section with an illustration of a rigorous cover of the Feigenbaum Julia set, see Figure~\ref{fig:level_14_real_c}(b).
\begin{figure}[h]
\begin{center}
\includegraphics[scale=0.55]{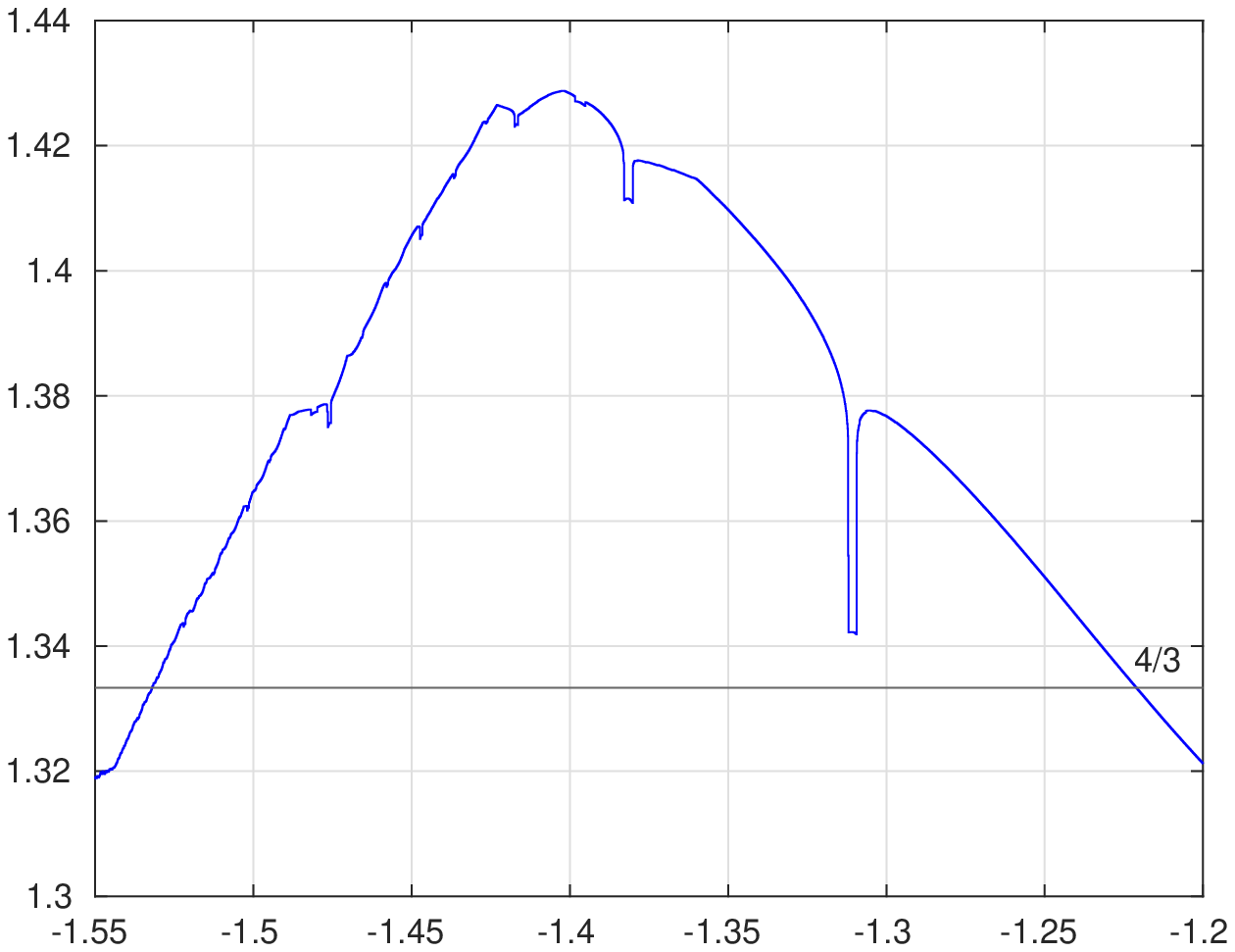}
\hspace*{-5mm}
\includegraphics[scale=0.55]{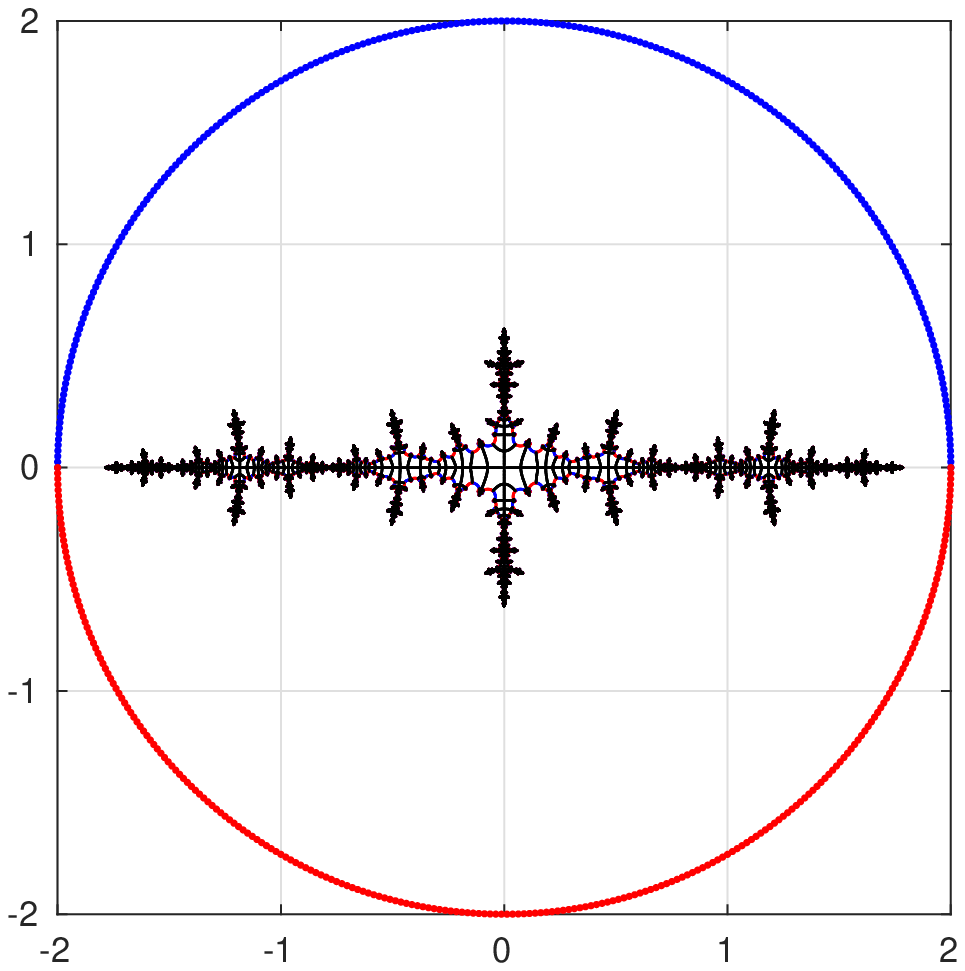}
\end{center}
\caption{(a) A graphical proof of Theorem~\ref{thm:uniform_bound_on_hd}.  (b) The level $k=14$ cover of the Feigenbaum Julia set.}
\label{fig:level_14_real_c}
\end{figure}

\appendix

\section{Technical statements}

Here we collect several technical statements and their proofs.

\begin{lemma}\label{Koebe1}
	Let $U\subset\hat{\C}$ be a topological disk and assume that $f\colon U\to\hat{\C}$ is a univalent map. If there exists $r>0$, such that each of the sets $\hat{\C}\setminus U$ and $\hat{\C}\setminus f(U)$ contains a spherical disk of radius $r$, then for any open subset $\tilde U\Subset U$, compactly contained in $U$, the distortion of $f$ on $\tilde U$, measured in the spherical metric, is bounded by some constant $K\ge 1$. This constant depends on the domain $\tilde U$ and the real number $r$.	
\end{lemma}
\begin{proof}
	After precomposing and postcomposing $f$ with appropriate isometries of the sphere, we may assume that both $U$ and $f(U)$ are contained in a Euclidean disk of radius $\rho(r)$, centered at zero. Here $\rho(r)$ is some increasing function of $r$. The latter implies that the restrictions of the spherical and the Euclidean metrics on $U\cup f(U)$ are equivalent with the equivalence constant depending on $r$. Now the statement of the lemma follows from the standard Koebe Distortion Theorem.	
\end{proof}

\begin{lemma}\label{Koebe2}
	Let $\mathcal F\colon\mathcal U\to\mathcal F(\mathcal U)$ be a partitioned holomorphic dynamical system that consists of at least two maps $f_j$ as in~(\ref{hol_sys_maps_eq}). Then there exists $r>0$, such that for any open spherical disk $D\subset\mathcal U\setminus PC(\mathcal F)\subset \hat{\C}$ of radius $r$, a positive integer $k\in\N$ and any connected component $D'$ of $\mathcal F^{-k}(D)$, the sets $\hat{\C}\setminus D$ and $\hat{\C}\setminus D'$ contain disks of radius $r$.
\end{lemma}

\begin{proof}
	Observe that each of the topological disks $D$ and $D'$ must be contained in one of the (possibly different) domains $U_j$. If there are at least two such domains, then after choosing a sufficiently small $r$, the statement becomes obvious.
\end{proof}

\begin{proof}[\bf Proof of Lemma \ref{concavity_lemma}] Let $M=\{m_{i,j}\}_{1\leqslant i,j\leqslant n}$ be a $n\times n$ primitive matrix and $M(\delta)=\{m_{i,j}^\delta\}_{1\leqslant i,j\leqslant n}$, $\delta \in\mathbb R$. We need to show that the spectral radius $\varrho(M(\delta))$ is a concave up function on $\mathbb R$.

 Let $m_k(\delta)$ be the entry in the first column and the first row of $(M(\delta))^k$. From Perron-Frobenius Theorem one can deduce that $(m_k(\delta))^{\frac{1}{k}}\to \varrho(\delta)$ when $k\to \infty$. Observe that for every $k$ one has
$$
m_k(\delta)=\sum\limits_{j=1}^r c_j^\delta,
$$
 for some $r\in\mathbb Z_+$ and $c_1,\ldots c_r>0$ depending on $k$. Let us show that $(m_k(\delta))^{\frac{1}{k}}$ is a concave up function for every $k$.

On has:
\begin{align*}\frac{d}{d\delta}(m_k(\delta))^{\frac{1}{k}}=
\frac{d}{d\delta}\left(\sum c_j^\delta\right)^{\frac{1}{k}}=\frac{\sum\log c_j\cdot c_j^\delta}{k\left(\sum c_j^\delta\right)^{\frac{k-1}{k}}}
\end{align*}
The second derivative of $(m_k(\delta))^\frac{1}{k}$ is equal to:
\begin{align*}\frac{d^2}{d\delta^2}\left(\sum c_j^\delta\right)^{\frac{1}{k}}=\frac{\left(\sum(\log c_j)^2\cdot c_j^\delta\right)\sum c_j^\delta-\frac{k-1}{k}\left(\sum\log c_j\cdot c_j^\delta\right)^2}{k\left(\sum c_j^\delta\right)^{\frac{2k-1}{k}}}.
\end{align*}
\begin{align*}\frac{d^2}{d\delta^2}\left(\sum c_j^\delta\right)^{\frac{1}{k}}=\frac{\left(\sum(\log c_j)^2\cdot c_j^\delta\right)\sum c_j^\delta-\frac{k-1}{k}\frac{\left(\sum\log c_j\cdot c_j^\delta\right)^2}{\left(\sum c_j^\delta\right)^{\frac{1}{k}}}}{k\left(\sum c_j^\delta\right)^{\frac{2k-2}{k}}}.
\end{align*}

By Cauchy-Schwarz inequality one has:
$$\left(\sum(\log c_j)^2\cdot c_j^\delta\right)\sum c_j^\delta\geqslant \left(\sum|\log c_j|\cdot c_j^\delta\right)^2.$$ We obtain that $\frac{d^2}{d\delta^2}(m_k(\delta))^\frac{1}{k}\geqslant 0$ and thus $(m_k(\delta))^\frac{1}{k}$ is concave up. Taking the limit when $k\to\infty$ we obtain that $\varrho(M(\delta))$ is concave up.
\end{proof}

\paragraph*{\bf Rigorous approximations of real Feigenbaum parameters.}
   Our computations rely on the theory of kneading sequences (see~\cite{Milnor_Thurston,Collet_Eckmann_book_80}). Given a unimodal map $f(z)$ with the critical point at $0$ the kneading sequence is the sequence $K(f)=\{k_1,k_2,\ldots\}\subset \{-1,0,1\}^\mathbb N$, where $k_i=-1$ if $f^i(0)<0$, $k_i=1$ if $f^i(0)>0$ and $k_i=0$ if $f^i(0)=0$. Given a Feigenbaum map $p_c(z)=z^2+c,c\in\mathbb R,$ with a stationary combinatorics of period $n$ the map $p_c^n(z)$ is hybrid equivalent to $p_c(z)$ near the origin. In particular, there is a homeomorphism $h$ from a segment $I\subset R$ containing the postcritical set of $f_c$ to a segment $I_1$ containing zero such that $p_c^n=h\circ p_c\circ h^{-1}$. From this we conclude that the kneading sequence $K(p_c)$ has the properties 
   \begin{equation}\label{EqKneadingSeqCalc} k_{n^lm}=(k_n/k_1)^lk_m=(-k_n)^lk_m,\;\;k_{s+n}=k_s\;\;\text{if}\;\;n\;\;\text{does not divide}\;\;s.
   \end{equation} 
This allows us to compute any number of elements of the kneading sequence $K(p_c)$ given $k_1,\ldots,k_n$.

   Notice that not all sequences $k_1,\ldots,k_n$ of length $n$ are realizable as an initial segment of a kneading sequence. To find all realizable sequences we can start by finding all real superattracting parameters of period $n$. Given a real superattracting parameter $a$ of period $n$, it follows from~\cite{McMullen_book2,Lyubich_hairiness} that there exists a unique period $n$ real Feigenbaum parameter $c=c(a)$ such that $k_i(c)=k_i(a)$ for $i=1,\ldots,n-1$ and $k_n(c)=k_n(a-\epsilon)$ for a sufficiently small $\epsilon>0$. The value of $k_n(c)$ can also be determined as $k_n(c) = -\mathrm{sign}\left(\frac{dp_x^n(0)}{dx}|_{x=a}\right)$.

   Given a point $b\in\mathbb R$, $b\neq c$, one can find out whether $b<c$ by comparing the kneading sequences of $p_b$ and $p_c$. Namely, let $i$ be the smallest number such that $k_i(b)\neq k_i(c)$. If $\mathrm{sign}(k_i(b))\neq\mathrm{sign}\left(\frac{dp_x^i(0)}{dx}|_{x=b}\right)$ then $b<c$, otherwise $b>c$.

   Now, from the above considerations, one can compute an approximation of $c$ by repeatedly cutting a segment containing $c$ into halves, starting
    with the segment $[-2,0]$. Using interval-valued numerics one can control the accumulated errors and ensure that the approximations of the Feigenbaum parameters in Theorem~\ref{thm:lower_bound_on_hd_c} are obtained with appropriate accuracy. %within the declared bounds.

\bibliographystyle{amsalpha}
\bibliography{HDbiblio}

\begin{bibdiv}
\begin{biblist}

\bib{AvilaLyubich-SmallHD-06}{incollection}{
      author={Avila, Artur},
      author={Lyubich, Mikhail},
       title={Examples of {F}eigenbaum {J}ulia sets with small {H}ausdorff
  dimension},
        date={2006},
   booktitle={Dynamics on the {R}iemann sphere},
   publisher={Eur. Math. Soc., Z\"{u}rich},
       pages={71\ndash 87},
         url={https://doi.org/10.4171/011-1/3},
      review={\MR{2348955}},
}

\bib{Avila_Lyubich_08}{article}{
      author={Avila, Artur},
      author={Lyubich, Mikhail},
       title={Hausdorff dimension and conformal measures of {F}eigenbaum
  {J}ulia sets},
        date={2008},
        ISSN={0894-0347},
     journal={J. Amer. Math. Soc.},
      volume={21},
      number={2},
       pages={305\ndash 363},
         url={https://doi.org/10.1090/S0894-0347-07-00583-8},
      review={\MR{2373353}},
}

\bib{AvilaLyubich-Area-22}{article}{
      author={Avila, Artur},
      author={Lyubich, Mikhail},
       title={Lebesgue measure of {F}eigenbaum {J}ulia sets},
        date={2022},
        ISSN={0003-486X},
     journal={Ann. of Math. (2)},
      volume={195},
      number={1},
       pages={1\ndash 88},
         url={https://doi.org/10.4007/annals.2022.195.1.1},
      review={\MR{4358413}},
}

\bib{Buff_Cherita_pos_area}{article}{
      author={Buff, Xavier},
      author={Ch\'{e}ritat, Arnaud},
       title={Quadratic {J}ulia sets with positive area},
        date={2012},
        ISSN={0003-486X},
     journal={Ann. of Math. (2)},
      volume={176},
      number={2},
       pages={673\ndash 746},
         url={https://doi.org/10.4007/annals.2012.176.2.1},
      review={\MR{2950763}},
}

\bib{Bowen_79}{article}{
      author={Bowen, Rufus},
       title={Hausdorff dimension of quasicircles},
        date={1979},
        ISSN={0073-8301},
     journal={Inst. Hautes \'{E}tudes Sci. Publ. Math.},
      number={50},
       pages={11\ndash 25},
         url={http://www.numdam.org/item?id=PMIHES_1979__50__11_0},
      review={\MR{556580}},
}

\bib{Collet_Eckmann_book_80}{book}{
      author={Collet, Pierre},
      author={Eckmann, Jean-Pierre},
       title={Iterated maps on the interval as dynamical systems},
      series={Progress in Physics},
   publisher={Birkh\"{a}user, Boston, Mass.},
        date={1980},
      volume={1},
        ISBN={3-7643-3026-0},
      review={\MR{613981}},
}

\bib{Co42}{article}{
      author={Collatz, L.},
       title={Einschliessungssatz f\"{u}r die charakteristischen {Z}ahlen von
  {M}atrizen},
        date={1942},
        ISSN={0025-5874},
     journal={Math. Z.},
      volume={48},
       pages={221\ndash 226},
         url={https://doi.org/10.1007/BF01180013},
      review={\MR{8590}},
}

\bib{DGM-HD-20}{misc}{
      author={Dobbs, Neil},
      author={Graczyk, Jacek},
      author={Mihalache, Nicolae},
       title={Hausdorff dimension of {J}ulia sets in the logistic family},
        date={2020},
        note={arXiv: 2007.07661},
}

\bib{DS-20}{article}{
      author={Dudko, Artem},
      author={Sutherland, Scott},
       title={On the {L}ebesgue measure of the {F}eigenbaum {J}ulia set},
        date={2020},
        ISSN={0020-9910},
     journal={Invent. Math.},
      volume={221},
      number={1},
       pages={167\ndash 202},
         url={https://doi.org/10.1007/s00222-020-00949-8},
      review={\MR{4105087}},
}

\bib{Dudko_20}{article}{
      author={Dudko, Artem},
       title={On {L}ebesgue measure and {H}ausdorff dimension of {J}ulia sets
  of real periodic points of renormalization},
        date={2020},
        ISSN={0239-7269},
     journal={Bull. Pol. Acad. Sci. Math.},
      volume={68},
      number={2},
       pages={151\ndash 168},
         url={https://doi.org/10.4064/ba210406-10-4},
      review={\MR{4246926}},
}

\bib{Graczyk_Smirnov_09}{article}{
      author={Graczyk, Jacek},
      author={Smirnov, Stanislav},
       title={Non-uniform hyperbolicity in complex dynamics},
        date={2009},
        ISSN={0020-9910},
     journal={Invent. Math.},
      volume={175},
      number={2},
       pages={335\ndash 415},
         url={https://doi.org/10.1007/s00222-008-0152-8},
      review={\MR{2470110}},
}

\bib{Graczyk_Smirnov_98}{article}{
      author={Graczyk, Jacek},
      author={Smirnov, Stas},
       title={Collet, {E}ckmann and {H}\"{o}lder},
        date={1998},
        ISSN={0020-9910},
     journal={Invent. Math.},
      volume={133},
      number={1},
       pages={69\ndash 96},
         url={https://doi.org/10.1007/s002220050239},
      review={\MR{1626469}},
}

\bib{JP-02}{article}{
      author={Jenkinson, Oliver},
      author={Pollicott, Mark},
       title={Calculating {H}ausdorff dimensions of {J}ulia sets and {K}leinian
  limit sets},
        date={2002},
        ISSN={0002-9327},
     journal={Amer. J. Math.},
      volume={124},
      number={3},
       pages={495\ndash 545},
  url={http://muse.jhu.edu/journals/american_journal_of_mathematics/v124/124.3jenkinson.pdf},
      review={\MR{1902887}},
}

\bib{JZ20}{article}{
      author={Jaksztas, Ludwik},
      author={Zinsmeister, Michel},
       title={On the derivative of the {H}ausdorff dimension of the {J}ulia
  sets for {$z^2+c$}, {$c\in\Bbb{R}$} at parabolic parameters with two petals},
        date={2020},
        ISSN={0001-8708},
     journal={Adv. Math.},
      volume={363},
       pages={106981, 59},
         url={https://doi.org/10.1016/j.aim.2020.106981},
      review={\MR{4054050}},
}

\bib{KapelaMrozekWilczakZgliczynski2020}{article}{
      author={Kapela, Tomasz},
      author={Mrozek, Marian},
      author={Wilczak, Daniel},
      author={Zgliczyński, Piotr},
       title={{CAPD::DynSys}: a flexible {C++} toolbox for rigorous numerical
  analysis of dynamical systems},
        date={2020},
     journal={Commun. Nonlinear Sci. Numer. Simul.},
       pages={105578},
}

\bib{Kozlovski_vStrien_09}{article}{
      author={Kozlovski, Oleg},
      author={van Strien, Sebastian},
       title={Local connectivity and quasi-conformal rigidity of
  non-renormalizable polynomials},
        date={2009},
        ISSN={0024-6115},
     journal={Proc. Lond. Math. Soc. (3)},
      volume={99},
      number={2},
       pages={275\ndash 296},
         url={https://doi.org/10.1112/plms/pdn055},
      review={\MR{2533666}},
}

\bib{Lyubich_hairiness}{article}{
      author={Lyubich, Mikhail},
       title={Feigenbaum-{C}oullet-{T}resser universality and {M}ilnor's
  hairiness conjecture},
        date={1999},
        ISSN={0003-486X},
     journal={Ann. of Math. (2)},
      volume={149},
      number={2},
       pages={319\ndash 420},
         url={https://doi.org/10.2307/120968},
      review={\MR{1689333}},
}

\bib{Levin_Zin_13}{article}{
      author={Levin, Genadi},
      author={Zinsmeister, Michel},
       title={On the {H}ausdorff dimension of {J}ulia sets of some real
  polynomials},
        date={2013},
        ISSN={0002-9939},
     journal={Proc. Amer. Math. Soc.},
      volume={141},
      number={10},
       pages={3565\ndash 3572},
         url={https://doi.org/10.1090/S0002-9939-2013-11660-6},
      review={\MR{3080178}},
}

\bib{McMullen-HDII-00}{article}{
      author={McMullen, Curtis~T.},
       title={Hausdorff dimension and conformal dynamics. {II}. {G}eometrically
  finite rational maps},
        date={2000},
        ISSN={0010-2571},
     journal={Comment. Math. Helv.},
      volume={75},
      number={4},
       pages={535\ndash 593},
         url={https://doi.org/10.1007/s000140050140},
      review={\MR{1789177}},
}

\bib{McMullen_book2}{book}{
      author={McMullen, Curtis~T.},
       title={Renormalization and 3-manifolds which fiber over the circle},
      series={Annals of Mathematics Studies},
   publisher={Princeton University Press, Princeton, NJ},
        date={1996},
      volume={142},
        ISBN={0-691-01154-0; 0-691-01153-2},
         url={https://doi.org/10.1515/9781400865178},
      review={\MR{1401347}},
}

\bib{McMullen-HDIII-98}{article}{
      author={McMullen, Curtis~T.},
       title={Hausdorff dimension and conformal dynamics. {III}. {C}omputation
  of dimension},
        date={1998},
        ISSN={0002-9327},
     journal={Amer. J. Math.},
      volume={120},
      number={4},
       pages={691\ndash 721},
  url={http://muse.jhu.edu/journals/american_journal_of_mathematics/v120/120.4mcmullen.pdf},
      review={\MR{1637951}},
}

\bib{Moore1966}{book}{
      author={Moore, Ramon~E.},
       title={Interval analysis},
   publisher={Prentice-Hall Englewood Cliffs},
        date={1966},
      volume={4},
}

\bib{Milnor_Thurston}{incollection}{
      author={Milnor, John},
      author={Thurston, William},
       title={On iterated maps of the interval},
        date={1988},
   booktitle={Dynamical systems ({C}ollege {P}ark, {MD}, 1986--87)},
      series={Lecture Notes in Math.},
      volume={1342},
   publisher={Springer, Berlin},
       pages={465\ndash 563},
         url={https://doi.org/10.1007/BFb0082847},
      review={\MR{970571}},
}

\bib{Nu86}{article}{
      author={Nussbaum, Roger~D.},
       title={Convexity and log convexity for the spectral radius},
        date={1986},
        ISSN={0024-3795},
     journal={Linear Algebra Appl.},
      volume={73},
       pages={59\ndash 122},
         url={https://doi.org/10.1016/0024-3795(86)90233-8},
      review={\MR{818894}},
}

\bib{ComplexBook}{book}{
      author={Petkovi\'{c}, Miodrag~S.},
      author={Petkovi\'{c}, Ljiljana~D.},
       title={Complex interval arithmetic and its applications},
      series={Mathematical Research},
   publisher={Wiley-VCH Verlag Berlin GmbH, Berlin},
        date={1998},
      volume={105},
        ISBN={3-527-40134-2},
      review={\MR{1674721}},
}

\bib{PRS_2004}{article}{
      author={Przytycki, Feliks},
      author={Rivera-Letelier, Juan},
      author={Smirnov, Stanislav},
       title={Equality of pressures for rational functions},
        date={2004},
        ISSN={0143-3857},
     journal={Ergodic Theory Dynam. Systems},
      volume={24},
      number={3},
       pages={891\ndash 914},
         url={https://doi.org/10.1017/S0143385703000385},
      review={\MR{2062924}},
}

\bib{Prz22}{article}{
      author={Przytycki, Feliks},
       title={Mcmullen’s and geometric pressures and approximating
  {H}ausdorff dimension of {J}ulia sets from below},
        date={2021},
     journal={Bulletin Polish Acad. Sci. Math.},
      volume={69},
       pages={115\ndash 137},
}

\bib{Prz_99}{article}{
      author={Przytycki, Feliks},
       title={Conical limit set and {P}oincar\'{e} exponent for iterations of
  rational functions},
        date={1999},
        ISSN={0002-9947},
     journal={Trans. Amer. Math. Soc.},
      volume={351},
      number={5},
       pages={2081\ndash 2099},
         url={https://doi.org/10.1090/S0002-9947-99-02195-9},
      review={\MR{1615954}},
}

\bib{Pollicott_Vytnova_22}{misc}{
      author={Pollicott, Mark},
      author={Vytnova, Polina},
       title={Hausdorff dimension estimates applied to lagrange and markov
  spectra, {Z}aremba theory, and limit sets of fuchsian groups},
        date={2022},
        note={arXiv: 2012.07083},
}

\bib{Shishikura-HDM-98}{article}{
      author={Shishikura, Mitsuhiro},
       title={The {H}ausdorff dimension of the boundary of the {M}andelbrot set
  and {J}ulia sets},
        date={1998},
        ISSN={0003-486X},
     journal={Ann. of Math. (2)},
      volume={147},
      number={2},
       pages={225\ndash 267},
         url={https://doi.org/10.2307/121009},
      review={\MR{1626737}},
}

\bib{Tucker2011book}{book}{
      author={Tucker, Warwick},
       title={Validated numerics},
   publisher={Princeton University Press, Princeton, NJ},
        date={2011},
        ISBN={978-0-691-14781-9},
        note={A short introduction to rigorous computations},
      review={\MR{2807595}},
}

\bib{Urbanski_94}{article}{
      author={Urba\'{n}ski, Mariusz},
       title={Rational functions with no recurrent critical points},
        date={1994},
        ISSN={0143-3857},
     journal={Ergodic Theory Dynam. Systems},
      volume={14},
      number={2},
       pages={391\ndash 414},
         url={https://doi.org/10.1017/S0143385700007926},
      review={\MR{1279476}},
}

\end{biblist}
\end{bibdiv}

\end{document}